\documentclass[12pt,oneside,a4paper,reqno]{amsart}
\usepackage[T2A,T1]{fontenc}
\usepackage[utf8]{inputenc}
\usepackage[top=3cm,right=1cm,bottom=3cm,left=3cm]{geometry}

\usepackage[english]{babel}

\usepackage{amsthm}
\usepackage{amsbsy}
\usepackage{amstext}
\usepackage{amssymb}

\usepackage{hyperref}
\usepackage{cite}

\theoremstyle{plain}
\newtheorem{lemma}{Lemma}
\newtheorem{proposition}{Proposition}
\newtheorem{corollary}{Corollary}
\newtheorem{theorem}{Theorem}
\theoremstyle{remark}
\newtheorem{remark}{Remark}

\title[Dynamical bifurcations]{Dynamical bifurcation of multi-frequency oscillations in a fast-slow system}
\author[A. M. Samoilenko]{A. M. Samoilenko$^{\ast}$}\thanks{$^{\ast}$ Institute of Mathematics NAS Ukraine}
\author[I. O. Parasyuk]{I. O. Parasyuk$^{\dag}$}\thanks{$^{\dag}$ National Taras Shevchenko Univesity of Kyiv}
\author[B. V. Repeta]{B. V. Repeta$^{\ddag}$}\thanks{$^{\ddag}$ National Taras Shevchenko Univesity of Kyiv}

\begin{document}
\begin{abstract}We study a dynamical counterpart of bifurcation to invariant torus
for a system of interconnected fast phase variables and slowly varying
parameters. We show that in such a system, due to the slow evolution
of parameters, there arise transient processes from damping oscillations
to multi-frequency ones, asymptotically close to motions on the invariant
torus.\end{abstract}
\maketitle

\section{Introduction}

It was shown in the monograph of Krylov~M.~M. and Bogoliubov~M.~M.~\cite{KryBog34} that non-conservative perturbations of a pair of harmonic oscillators under quite general conditions lead to the birth of a local attractor  homeomorphic to a 2-dimensional torus in a 4-dimensional phase space of such a system. The phenomenon of invariant torus bifurcation due to a stability loss of a limit cycle when a couple of complex multiplicators cross the unit circle as parameters change was examined in~\cite{Nei59, Sack65} and became widely known with appearance of  publications~\cite{RueTak71, MarMcM76} (see also~\cite{Kuz98}). The mathematical framework which enabled us to achieve strict results in the analysis of multidimensional invariant tori bifurcations was developed in~\cite{Hale61, BMS69, Fen71, SamMit76, Sam91, Sam97, SamPet04}. It is also worth highlighting references~\cite{Lang79, Gav80, Bib90, Bib91, Gol98, BibBuk12} among other works in this direction.

The mentioned results are related to the static bifurcations theory, which considers systems of the form $\dot{x}=f(x,u)$, dependent on time-constant parameters $u=(u_{1},\ldots,u_{m})$, where $f(\cdot,\cdot) \colon \mathbb{R}^{d} \times \mathbb{R}^{m} \to \mathbb{R}^{d}$ is a sufficiently smooth mapping. When one claims, for instance, that while the parameters $u$ change along some curve $u=u(s)$, $s \in (-1, 1)$, there is a stable $k$-dimensional invariant torus being born in such a system due to a stability loss of an equilibrium, it means that when $s \in (-1, 0)$, the system
\begin{equation}\label{eq:bwp_1}
	\dot{x} = f(x,u(s))
\end{equation}
has an asymptotically stable equilibrium $x_{\ast}$ which becomes unstable for $s \in (0,1)$, and at the same time there exists a continuous (sufficiently smooth) mapping $\mathfrak{X}(\cdot,\cdot) \colon  \mathbb{T}^{k} \times (-1,1) \to \mathbb{R}^{d}$ such that $\mathfrak{X}(\cdot,0) \equiv x_{\ast}$ for all $s \in (-1,0]$, and for each $s \in (0,1)$ the image of $\mathfrak{X}(\cdot,s) \colon \mathbb{T}^{k} \to \mathbb{R}^{d}$ is an asymptotically orbitally stable invariant toroidal manifold of system~\eqref{eq:bwp_1}.

Due to certain reasons, it sometimes makes sense to view the family of systems as the system in $\mathbb{R}^{d} \times \mathbb{R}^{m}$
\begin{equation}\label{eq:bwp_2}
	\dot{x} = f(x,u),\quad\dot{u}=0.
\end{equation}
Then one can interpret the bifurcation of invariant torus as follows. The system has an invariant set, such that its intersections with planes $\Pi_{s} :=  \left\{  (x,u) \colon u=u(s) \right\}$ for $s \in (0,1)$ are invariant toroidal manifolds $\mathcal{T}_{s} := \mathfrak{X}(\mathbb{T}^{k},s) \times  \left\{  u(s)\right\}$. These manifolds shrink to the point  $\left\{  (x_{\ast}, u(0)) \right\}$ as $s \to +0$, and each manifold $\mathcal{T}_{s}$ is a local attractor of system~\eqref{eq:bwp_2} restriction to the plane $\Pi_{s}$ (which is clearly  invariant). Thus, when one says that the aforementioned bifurcation consists in a birth of an invariant torus in the phase space, it should actually be perceived as a spatial phenomenon rather than a dynamical one.

With the release of publications~\cite{Shi73,Nei85,Nei87}, there began the systematic research of truly dynamical bifurcations --- the effects connected with qualitative changes in a system's dynamics that develop in time and are induced by the actual slow evolution of parameters as they pass certain critical values. One of the most resonant achievements in this direction was connected with the delayed loss of stability phenomenon in so-called fast-slow systems of the form
\begin{equation}\label{eq:slman1}
	\dot{x} = f(x,u,\varepsilon), \quad
	\dot{u} = \varepsilon g(x,u,\varepsilon),
\end{equation}
where $x=(x_{1}, \ldots, x_{d})$ are fast phase variables, $u=(u_{1}, \ldots, u_{m})$ are slowly varying parameters and $\varepsilon$ is a small static parameter. Here we cannot present the complete review of results in dynamical bifurcations theory. Let us just note some of the papers~\cite{Ben91,BuNeSch04,RachSch05,Ano05}. In particular,~\cite{Ben91} depicts the connection between the delayed loss of stability and ``canard''-solutions theory of singularly perturbed systems~\cite{Car84}, whereas in~\cite{RachSch05} a dynamical counterpart of the Andronov -- Hopf bifurcation was studied. Meanwhile, for authors' best knowledge, the information on dynamical counterparts of invariant tori bifurcations seems to be lacking.

In this paper we consider a $(2n+m)$-dimensional system~\eqref{eq:slman1} ($d=2n$) under assumption that its invariant manifold of slow motions (i.~m.~s.~m.) is given by equation $x=0$, i.~e. $f(0,u,\varepsilon)\equiv0$, and for the linear system
\begin{equation*}
	\dot{x} = \left[ f_{x}^{\prime}(0,u,0) + \varepsilon f_{x,\varepsilon}^{\prime\prime}(0,u,0) \right] x,
\end{equation*}
(which is the system of the first approximation for the phase variables $x$  with respect to the i.~m.~s.~m.) one can specify the following three zones in the parameters space: the zone of asymptotic stability $\mathcal{D}_{s}$, indefiniteness zone $\mathcal{D}_{\ast}$, and the zone of complete instability $\mathcal{D}_{u}$ . At the same time, the characteristic equation of the operator $f_{x}^{\prime}(0,u,0)$ has purely imaginary roots for all $u$ from  union of the aforementioned domains. In addition, we assume that the system $\dot{u} = \varepsilon g(0,u,0)$ is convergent and its attractor is some point in $\mathcal{D}_{u}$. Under certain additional conditions it will be shown that in an $O(\sqrt{\varepsilon})$-neighborhood of the i.~m.~s.~m. of system~\eqref{eq:slman1} one can observe the following dynamical bifurcation. Firstly, while during some time of order $O(\varepsilon^{-1})$ the parameters $u(t)$ move inside the zone $\mathcal{D}_{s}$, the phase components of the corresponding solution $x(t)$ exhibit exponentially damping oscillations. Next, after $u(t)$ has passed $\mathcal{D}_{\ast}$ and has entered $\mathcal{D}_{u}$, the oscillations' amplitude starts to grow, and finally, as $t\to+\infty$ the corresponding trajectory of~\eqref{eq:slman1} is attracted to the invariant torus, asymptotically approaching some trajectory on the latter.

To analyze the system in an $O(\sqrt{\varepsilon})$-neighborhood of an i.~m.~s.~m., the scaling transformation $x\mapsto\sqrt{\varepsilon}x$ is applied, after what the problem becomes a non-local one. The establishment of invariant torus' existence itself does not give rise to many essential complications and is done using the same results of \cite{Hale61,SamMit76} in quite the same manner as in~\cite{Bib91,BibBuk12}. However, it is a much harder task to determine  the non-local attraction basin of an invariant torus of the system obtained by the above scaling. For this we have managed to show that the relative measure of the attraction domain is estimated from below with a value of order $1-O(\varepsilon^{k/n})$.

The present article is organized as follows. In Section~\ref{sec:const-normal-sys} we formulate a series of conditions on the system under consideration, construct its partial normal form in phase variables $x$ and make a transition to polar-type variables. In Section~\ref{sec:first_approx} the behavior of solutions of the first approximation system is examined. Section~\ref{sec:main_theorem} contains our main result, which is based on auxiliary propositions of Sections~\ref{sec:pre_analysis} and~\ref{sec:existinvtor}, concerning the existence of an invariant torus and its attraction properties.

\section{Construction of System's Normal Form in Phase Variables\\ and Main Assumptions}\label{sec:const-normal-sys}

From now on we will require  system~\eqref{eq:slman1} to satisfy the following conditions.
\begin{description}
\item [{C1}] The right-hand sides of the system are smooth and bounded. Particularly, $f(\cdot,\cdot,\cdot) \in \mathrm{\hat{C}}^{\infty} \left( \mathbb{R}^{2n} \times \mathbb{R}^{m} \times \mathbb{R} \!\to\! \mathbb{R}^{2n} \right)$, $g(\cdot,\cdot,\cdot) \in \hat{\mathrm{C}}^{\infty} \left( \mathbb{R}^{2n} \times \mathbb{R}^{m} \times \mathbb{R} \to\! \mathbb{R}^{m} \right)$, where $\hat{\mathrm{C}}^{\infty} \left( \mathcal{X} \!\to\! \mathcal{Y} \right)$ denotes the space of smooth bounded mappings from domain $\mathcal{X}$ to set $\mathcal{Y}$ with bounded derivatives of all orders.
\item [{C2}] The system has an i.~m.~s.~m. given by equation $x=0$, i.~e. $f(0,u,\varepsilon)=0$ for all $(u,\varepsilon) \in \mathbb{R}^{m} \times \mathbb{R}$.
\item [{C3}] For every $u \in \mathbb{R}^{m}$ the operator $f_{x}^{\prime}(0,u,0)$ has purely imaginary eigenvalues $\pm\mathrm{i}\omega_{j}(u)$, $j=1,\ldots,n$, such that
\begin{equation*}
	\inf_{u \in \mathbb{R}^{m}} \omega_{j}(u) > 0, \quad \inf_{u\in\mathbb{R}^{m}} \left| \omega_{j}(u) - \omega_{k}(u) \right| > 0, \quad j, k = 1, \ldots ,n, \; j \ne k.
\end{equation*}
\end{description}

Then for all natural $N \ge 2$ and $s \ge 2$ we may express system ~\eqref{eq:slman1} as
\begin{equation}\label{eq:smb_fs}
	\dot{x} = \sum_{k=1}^{N} F_{k}(u,\varepsilon) x^{k} + \tilde{F}_{N,s+1}(x,u,\varepsilon) x, \quad
	\dot{u} = \varepsilon \left[ \sum_{k=0}^{N} G_{k}(u,\varepsilon) x^{k} + \tilde{G}_{N+1,s}(x,u,\varepsilon) \right].
\end{equation}
Here $F_{k}(u,\varepsilon)x^{k}$ and $G_{k}(u,\varepsilon)x^{k}$ are $\mathbb{R}^{2n}$- and $\mathbb{R}^{m}$-valued homogeneous forms of degree $k$ of $x$ and polynomials of degree $s$ and $s-1$ in $\varepsilon$ respectively. The remainder terms of Taylor's formula $\tilde{F}_{N,s+1}(x,u,\varepsilon)x$ (here $\tilde{F}_{N,s+1}(x,u,\varepsilon)$ is a $2n\times2n$-matrix) and $\tilde{G}_{N+1,s}(x,u,\varepsilon)$ for $\left\Vert x\right\Vert +\left|\varepsilon\right|\to0$ satisfy the order relations
\begin{equation*}
	\left\Vert \tilde{F}_{N,s+1}(x,u,\varepsilon) \right\Vert = O \left( \left\Vert x \right\Vert ^{N} + \left| \varepsilon \right|^{s+1} \right), \quad
	\left\Vert \tilde{G}_{N+1,s}(x,u,\varepsilon) \right\Vert = O \left( \left\Vert x \right\Vert ^{N+1} + \left| \varepsilon \right|^{s} \right).
\end{equation*}
Without loss of generality we can assume that for a fixed natural $s$ and the corresponding sufficiently small $\varepsilon_{0}>0$ the matrix $F_{1}(u,\varepsilon)$ is in its real normal form on the set $\mathbb{R}^{m} \times (-\varepsilon_{0}, \varepsilon_{0})$
\begin{equation*}
	J(u,\varepsilon) := \mathrm{diag} \left[ \begin{pmatrix}
		\varepsilon\bar{\alpha}_{1}(u,\varepsilon) & -\bar{\omega}_{1}(u,\varepsilon)\\
		\bar{\omega}_{1}(u,\varepsilon) & \varepsilon\bar{\alpha}_{1}(u,\varepsilon)
	\end{pmatrix},
	\ldots,
	\begin{pmatrix}
		\varepsilon\bar{\alpha}_{n}(u,\varepsilon) & -\bar{\omega}_{n}(u,\varepsilon)\\
		\bar{\omega}_{n}(u,\varepsilon) & \varepsilon\bar{\alpha}_{n}(u,\varepsilon)
	\end{pmatrix} \right].
\end{equation*}
Here each function $\bar{\alpha}_{j}(u,\varepsilon)$, $\bar{\omega}_{j}(u,\varepsilon)$ is a polynomial in $\varepsilon$ of degree not greater than $s-1$ and $s$ respectively with smooth coefficients depending on $u$ of class $\mathrm{\hat{C}}^{\infty} \left( \mathbb{R}^{m} \!\to\! \mathbb{R} \right)$, and $\bar{\omega}_{j}(u,0) = \omega_{j}(u)$. To verify this, the following lemma can be used.

\begin{lemma}
Suppose that $A(\cdot,\cdot) \in \hat{\mathrm{C}}^{\infty} \left( \mathbb{R}^{m} \times \mathbb{R} \!\to\! \mathbb{R}^{d \times d} \right)$, $G(\cdot,\cdot) \in \hat{\mathrm{C}}^{\infty} \left (\mathbb{R}^{m} \times \mathbb{R} \!\to\! \mathbb{R}^{m} \right)$, where $\mathbb{R}^{d \times d}$ denotes the space of $d \times d$-matrices with real elements. If for all $u \in \mathbb{R}^{m}$ the matrix $A_{0}(u) := A(u,0)$ has  eigenvalues $\lambda_{j}(u)$, $j=1,\ldots,d,$ such that
\begin{equation}\label{eq:slmaneig}
	\inf_{u  \in \mathbb{R}^{m}} \left| \lambda_{i}(u) - \lambda_{j}(u) \right| > 0 \quad \forall i,j=1,\ldots,d, \quad i\ne j,
\end{equation}
then for any natural $s$ there exists a mapping $T(\cdot,\cdot) \in \mathrm{C}^{\infty} \left( \mathbb{R}^{m} \times \mathbb{R} \!\to\! \mathbb{R}^{d \times d} \right)$ with the following properties. 1)~The mapping $\varepsilon \mapsto T(u,\varepsilon)$ is an $\mathbb{R}^{d \times d}$-valued polynomial of degree $s$ in $\varepsilon$ with coefficients of class  $\hat{\mathrm{C}}^{\infty} \left( \mathbb{R}^{m} \!\to\! \mathbb{R}^{d \times d} \right)$. 2)~There exists  $\varepsilon_{0}>0$ such that $\inf_{(u,\varepsilon) \in \mathbb{R}^{m} \times (-\varepsilon_{0}, \varepsilon_{0})} \left| \det T(u,\varepsilon) \right| > 0$, and after the transformation $x\mapsto T(u,\varepsilon)x$ the system
\begin{equation}\label{eq:lemma1}
	\dot{x} = A(u,\varepsilon) x, \quad
	\dot{u} = \varepsilon G(u,\varepsilon)
\end{equation}
takes the form
\begin{equation*}
	\dot{x} = \left[ B(u,\varepsilon) + \varepsilon^{s+1} \tilde{B}(u,\varepsilon) \right] x, \quad
	\dot{u} = \varepsilon G(u,\varepsilon),
\end{equation*}
where the matrix $B(u,\varepsilon) = \sum_{k=0}^{s} \varepsilon^{k} B_{k}(u)$ is in its real normal form, and
\begin{equation*}
	B_{k}(\cdot) \in \hat{\mathrm{C}}^{\infty} \left( \mathbb{R}^{m} \!\to\! \mathbb{R}^{d \times d}\right),\; k=0,\ldots,s, \quad
	\tilde{B}(\cdot,\cdot) \in \hat{\mathrm{C}}^{\infty} \left( \mathbb{R}^{m} \times (-\varepsilon_{0}, \varepsilon_{0}) \!\to\! \mathbb{R}^{d \times d} \right).
\end{equation*}
\end{lemma}

\begin{proof}
Since condition~\eqref{eq:slmaneig} is met, there exists a mapping $T_{0}(\cdot) \in \hat{\mathrm{C}}^{\infty} \left( \mathbb{R}^{m} \!\to\! \mathbb{R}^{d \times d} \right)$ such that the matrix $B_{0}(u) := T_{0}^{-1}(u)A_{0}(u)T_{0}(u)$ is in its real normal form. Moreover, there exists a constant matrix $S$, with complex elements in general, such that $S^{-1}B_{0}(u)S$ is diagonal. Let us construct a formal change of variables
\begin{equation*}
	x \mapsto \sum_{k\ge0} \varepsilon^{k} T_{k}(u) x
\end{equation*}
with the coefficients $T_{k}(\cdot) \in \hat{\mathrm{C}}^{\infty} \left( \mathbb{R}^{m} \!\to\! \mathbb{R}^{d \times d} \right)$ which transforms system~\eqref{eq:lemma1} into
\begin{equation*}
	\dot{x} = \sum_{k\ge0} \varepsilon^{k} B_{k}(u) x, \quad
	\dot{u} = \varepsilon G(u,\varepsilon),
\end{equation*}
where $B_{k}(\cdot) \in \hat{\mathrm{C}}^{\infty} \left( \mathbb{R}^{m} \!\to\! \mathbb{R}^{d \times d} \right)$ and $B_{k}(u)$ commutes with $B_{0}(u)$ for all $k \ge 1$.

To do so, let us introduce the dot product  $\left\langle X,Y \right\rangle := \mathrm{tr}(XY)$ in $\mathbb{R}^{d \times d}$ and note that as
\begin{equation*}
	\left\langle ZX - XZ, Y \right\rangle = \mathrm{tr} \left[ (ZX - XZ) Y \right] = \mathrm{tr} (YZX - ZYX) = -\left\langle X, ZY - YZ \right\rangle,
\end{equation*}
the operator $X \mapsto \mathrm{ad}_{Z}X := ZX - XZ$ is skew symmetric  for all $Z \in \mathbb{R}^{d \times d}$, thus, $\mathbb{R}^{d \times d} = \ker \mathrm{ad}_{Z} \oplus \mathrm{im} \, \mathrm{ad}_{Z}$ (the sum is orthogonal and $\mathrm{ad}_{Z}$-invariant). Then, if for any arbitrary $Y \in \mathbb{R}^{d \times d}$ we denote its orthogonal projection on $\ker \mathrm{ad}_{Z}$ by $Y_{0}$, the equation $\mathrm{ad}_{Z}X = Y - Y_{0}$ with a  fixed $Z$ has a unique solution $X \in \mathrm{im} \, \mathrm{ad}_{Z}$. In addition to this, if $Z$ has $N$ different eigenvalues, then there exists a non-degenerate matrix $S$ (with complex elements) such that $S^{-1} Z S$ is a diagonal matrix. Therefore, $X \in \ker \mathrm{ad}_{Z}$ if and only if $S^{-1} X S$ is diagonal.

Next, let $\sum_{i \ge 0} \varepsilon^{i} A_{i}(u)$ and $\sum_{j \ge 0} \varepsilon^{j} G_{j}(u)$ be formal expansions of $A(u,\varepsilon)$ and $G(u,\varepsilon)$ by $\varepsilon$ respectively.  Equating coefficients of powers of $\varepsilon$ in the formal equality
\begin{equation*}
	\varepsilon \sum_{i \ge 0} \varepsilon^{i} \frac{\partial T_{i}}{\partial u} \sum_{j \ge 0} \varepsilon^{j} G_{j}(u) +
	\sum_{i \ge 0} \varepsilon^{i} T_{i}(u) \sum_{j \ge 0} \varepsilon^{j} B_{k}(u) =
	\sum_{i \ge 0} \varepsilon^{i} A_{i}(u) \sum_{j \ge 0} \varepsilon^{j} T_{j}(u),
\end{equation*}
which is satisfied by $B_{k}(u)$ and $T_{k}(u)$, we obtain relations
\begin{equation*}
\begin{gathered}
	T_{0}(u)B_{0}(u) = A_{0}(u)T_{0}(u),\\
	T_{k}(u)B_{0}(u) + T_{0}(u)B_{k}(u) = A_{0}(u)T_{k}(u) + A_{k}(u)T_{0}(u) + R_{k}(u), \quad k=1,2,\ldots,
\end{gathered}
\end{equation*}
in which every operator $R_{k}(u)$ is defined by $T_{i}(u)$, $B_{j}(u)$, $A_{l}(u)$ with indices less than $k$. If we assign $T_{k}(u) := T_{0}(u)X_{k}(u)$ for $k \ge 1$ and multiply all of the equalities by $T_{0}^{-1}(u)$ on the left side, we will have
\begin{equation*}
\begin{gathered}
	B_{0}(u) = T_{0}^{-1}(u)A_{0}(u)T_{0}(u),\\
	-\mathrm{ad}_{B_{0}(u)}X_{k}(u) = T_{0}^{-1}(u)A_{k}(u)T_{0}(u) + R_{k}(u)-B_{k}(u), \quad k \ge 1.
\end{gathered}
\end{equation*}
Now it is possible to explicitly determine $X_{k}(u) \in \mathrm{im} \, \mathrm{ad}_{B_{0}(u)}$ by replacing $B_{k}(u)$ with the orthogonal projection of matrix $P_{k}(u) := T_{0}^{-1}(u)A_{k}(u)T_{0}(u) + R_{k}(u)$ on $\ker \mathrm{ad}_{B_{0}(u)}$. (It follows from what was mentioned before that $B_{k}(u)$ equals the diagonal part of matrix $S^{-1}P_{k}(u)S$ multiplied by $S$ on its left side and by $S^{-1}$ on its right side).

Clearly, the desired non-formal transformation is $T(u, \varepsilon) = \sum_{k=0}^{s} \varepsilon^{k} T_{k}(u)$.
\end{proof}

Set $s_{j} \in \mathbb{C}^{2n}$ to be an eigenvector of matrix $J_{0}(u) := F_{1}(u,0)=J(u,0)$ which corresponds to the eigenvalue $\mathrm{i}\omega_{j}(u)$, $k=1,\ldots,n$. Since
\begin{equation*}
	J_{0}(u) = \mathrm{diag}\left[ \begin{pmatrix}
		0 & -\omega_{1}(u)\\
		\omega_{1}(u) & 0
	\end{pmatrix},
	\ldots,
	\begin{pmatrix}
		0 & -\omega_{n}(u)\\
		\omega_{n}(u) & 0
	\end{pmatrix} \right],
\end{equation*}
the vectors $s_{j}$ are independent of $u$. Let us compose a matrix $S$, whose first $n$ columns are the vectors $s_{1},\ldots,s_{n}$ and last $n$ columns are their complex conjugates respectively, and let us define basis forms
\begin{equation}\label{eq:bwp_basform}
	\varsigma_{\mathbf{q}}(y) := [S^{-1}y]^{\mathbf{q}}, \quad
	e_{i,\mathbf{q}}(y) = \varsigma_{\mathbf{q}}(y)s_{i},
\end{equation}
where $\mathbf{q} := (q_{1}, \ldots, q_{2n}) \in \mathbb{Z}_{+}^{2n}$, $x^{\mathbf{q}} = x_{1}^{q_{1}} \cdots  x_{2n}^{q_{2n}}$.

We can now proceed to the construction of a transformation that converts the $N$-jet of system~\eqref{eq:smb_fs} to its normal form in fast variables under an extra assumption of the absence of resonances up to a certain order between frequencies $\omega_{k}(u)$, $k=1,\ldots,n$. To make the corresponding statement, we define an $n\times2n$-matrix $I=[E_{n};-E_{n}]$, where $E_{n}$ is the $n$-dimensional identity matrix, assign
\begin{equation*}
	\omega(u) := \left( \omega_{1}(u), \ldots, \omega_{n}(u) \right), \quad
	\left| \mathbf{q} \right| := \left| q_{1} \right| + \cdots + \left| q_{2n} \right|,
\end{equation*}
denote the $i$-th unit vector of the coordinate space $\mathbb{R}^{2n}$ (i.~e. the vector, whose $i$-th coordinate is equal to $1$ and the rest are zeroes) by $\mathbf{e}_{i}$, and introduce the following sets for positive numbers $\nu$ and $\sigma$
\begin{equation*}
\begin{gathered}
	\mathcal{A}_{i}(N,\nu) :=  \left\{  u \in \mathbb{R}^{m} \colon \; \left| \left\langle \omega(u), I(\mathbf{q} - \mathbf{e}_{i}) \right\rangle \right| > \nu \; \forall \mathbf{q} \in \mathbb{Z}_{+}^{2n} \colon 2 \le \left| \mathbf{q} \right| \le N, \; I (\mathbf{q} - \mathbf{e}_{i}) \ne 0 \right\},\\
	\mathcal{A}_{0}(N,\nu) :=  \left\{  u \in \mathbb{R}^{m} \colon \; \left| \left\langle  \omega(u), I\mathbf{q} \right\rangle \right| > \nu \; \forall \mathbf{q} \in \mathbb{Z}_{+}^{2n} \colon 2 \le \left| \mathbf{q} \right| \le N, \; I\mathbf{q} \ne 0 \right\},\\
	\mathcal{A}(N,\nu) := \bigcap_{i=0}^{n} \mathcal{A}_{i}(N,\nu), \quad
	B_{\delta}^{2n}(y_{0}) :=  \left\{  y \colon \left\Vert y - y_{0} \right\Vert < \delta \right\}, \quad
	B_{\delta}^{2n} := B_{\delta}^{2n}(0),\\
	\mathfrak{R}_{0}(N) :=  \left\{  \mathbf{q} \in \mathbb{Z}_{+}^{2n} \colon 0 \le \left| \mathbf{q} \right| \le N, \;  I\mathbf{q} = 0 \right\}, \;
	\mathfrak{R}_{i}(N) :=  \left\{  \mathbf{q} \in \mathbb{Z}_{+}^{2n} \colon 2 \le \left| \mathbf{q} \right| \le N, \;  I (\mathbf{q} - \mathbf{e}_{i}) = 0 \right\} .
\end{gathered}
\end{equation*}

\begin{proposition}\label{prop:bwp_1}
Suppose that conditions~\textbf{C1}--\textbf{C3} are met, and $s \ge 2$ is a fixed natural number. Also, for some $N \in \mathbb{N}$, $N>2$ and $\nu>0$ let the set $\mathcal{A}(N,\nu)$ be non-empty. Then there exist numbers $\delta>0$ and $\varepsilon_{0}>0$ such that after the change of variables
\begin{equation}\label{eq:bwp_ch_v}
	x = y + \sum_{k=2}^{N} X_{k}(v,\varepsilon) y^{k}, \quad
	u = v + \varepsilon \sum_{k=1}^{N} U_{k}(v,\varepsilon) y^{k}, \quad
	(y, v, \varepsilon) \in B_{\delta}^{2n} \times \mathcal{A}(N,\nu) \times (-\varepsilon_{0},\varepsilon_{0}),
\end{equation}
where
\begin{equation*}
	X_{k}(v, \varepsilon) y^{k} = \sum_{j=0}^{s} \varepsilon^{j} X_{k,j}(v) y^{k}, \quad
	U(v, \varepsilon) = \sum_{j=0}^{s-1} \varepsilon^{j} U_{k,j}(v) y^{k},
\end{equation*}
and
\begin{equation*}
	X_{k,j}(\cdot) y^{k} \in \hat{\mathrm{C}}^{\infty} \left( \mathcal{A}(N,\nu) \!\to\! \mathbb{R}^{2n} \right), \;
	U_{k,j}(\cdot) y^{k} \in \hat{\mathrm{C}}^{\infty} \left( \mathcal{A}(N,\nu) \!\to\! \mathbb{R}^{m} \right) \;
	\forall y \in \mathbb{R}^{2n}, \; k = 1, \ldots, s,
\end{equation*}
system~\eqref{eq:smb_fs} takes the form
\begin{equation}\label{eq:bwp_nf1}
\begin{aligned}
	\dot{y} &= J(v,\varepsilon) y + \sum_{i=1}^{2n} \sum_{\mathbf{q} \in \mathfrak{R}_{k}(N)} H_{i,\mathbf{q}}(v,\varepsilon) e_{i,\mathbf{q}}(y) + \tilde{H}_{N,s+1}(y,v,\varepsilon) y,\\
	\dot{v} &= \varepsilon \left[ \sum_{\mathbf{q} \in \mathfrak{R}_{0}(N)} \varsigma_{\mathbf{q}}(y) C_{\mathbf{q}}(v,\varepsilon) + \tilde{C}_{N+1,s}(y,v,\varepsilon) \right].
\end{aligned}
\end{equation}
Additionally,
\begin{equation*}
	H_{i,\mathbf{q}}(v,\varepsilon) = \sum_{j=0}^{s} \varepsilon^{j} H_{i,\mathbf{q},j}(v), \quad
	C_{\mathbf{q}}(v,\varepsilon) = \sum_{j=0}^{s-1} \varepsilon^{j} C_{\mathbf{q},j}(v),
\end{equation*}
where $H_{i,\mathbf{q},j}(\cdot) \in \hat{\mathrm{C}}^{\infty} \left( \mathcal{A}(N,\nu) \!\to\! \mathbb{C} \right), \;  C_{\mathbf{q},j}(\cdot) \in \hat{\mathrm{C}}^{\infty} \left( \mathcal{A}(N,\nu) \!\to\! \mathbb{C}^{m} \right)$, and the remainder terms of Taylor's formula $\tilde{H}_{N,s+1}(y,v,\varepsilon) y$ and $\tilde{C}_{N+1,s}(y,v,\varepsilon)$ satisfy the order relations
\begin{equation*}
	\left\Vert \tilde{H}_{N,s+1}(y,v,\varepsilon) \right\Vert = O \left( \left\Vert y \right\Vert ^{N} + \varepsilon^{s+1} \right), \quad
	\left\Vert \tilde{C}_{N+1,s}(y,v,\varepsilon) \right\Vert = O\left( \left\Vert y \right\Vert^{N+1} + \varepsilon^{s} \right), \quad
	\left\Vert y \right\Vert +\left| \varepsilon \right| \to 0.
\end{equation*}
\end{proposition}

\begin{proof}
Let us apply the polynomial transformation~\eqref{eq:bwp_ch_v} to the truncated system
\begin{equation*}
	\dot{x} = J(u,\varepsilon) x + \sum_{k=2}^{N} F_{k}(u,\varepsilon) x^{k}, \quad
	\dot{u} = \varepsilon \sum_{k=0}^{N} G_{k}(u,\varepsilon) x^{k},
\end{equation*}

If the resulting system is
\begin{equation*}
	\dot{y} = J(v,\varepsilon) y + \sum_{k \ge 2} H_{k}(v,\varepsilon) y^{k}, \quad
	\dot{v} = \varepsilon \sum_{k \ge 0} C_{k}(v,\varepsilon) y^{k},
\end{equation*}
then the following equalities must hold
\begin{equation*}
\begin{gathered}
	J(v,\varepsilon) y + \sum_{k \ge 2} H_{k}(v,\varepsilon) y^{k} + \sum_{j=2}^{N} \frac{\partial \left( X_{j}(v,\varepsilon) y^{j} \right)}{\partial y} \left[ J(v,\varepsilon) y + \sum_{i \ge 2} H_{i}(v,\varepsilon) y^{i} \right] +\\
	+ \sum_{j=2}^{N} \frac{\partial \left( X_{j}(v,\varepsilon) y^{j} \right)}{\partial v} \left[ \varepsilon \sum_{i \ge 0} C_{i}(v,\varepsilon) y^{i} \right] = \\
	= J \left( v + \varepsilon \sum_{j=1}^{N} U_{j}(v,\varepsilon) y^{j}, \varepsilon \right) \left[ y + \sum_{i=2}^{N} X_{i}(v,\varepsilon) y^{i} \right] + \\
	+ \sum_{j=2}^{N} F_{j} \left( v + \varepsilon \sum_{l=1}^{N} U_{l}(v,\varepsilon) y^{l}, \varepsilon \right) \left[ y + \sum_{i=2}^{N} X_{i}(v,\varepsilon) y^{i} \right]^{j},
\end{gathered}
\end{equation*}
\begin{equation*}
\begin{gathered}
	\sum_{k \ge 0} C_{k}(v,\varepsilon) y^{k} + \sum_{j=1}^{N} \frac{\partial \left[ U_{j}(v,\varepsilon) y^{j} \right]}{\partial y} \left[ J(v,\varepsilon) y + \sum_{i \ge 2} H_{i}(v,\varepsilon) y^{i} \right] + \\
	+ \sum_{j=1}^{N} \frac{\partial \left( U_{j}(v,\varepsilon) y^{j} \right)}{\partial v} \left[ \varepsilon \sum_{i \ge 0} C_{i}(v,\varepsilon) y^{i} \right] = \sum_{j=0}^{N} G_{j} \left( v + \varepsilon \sum_{l=1}^{N} U_{l}(v,\varepsilon) y^{l}, \varepsilon \right) \left[ y + \sum_{i=2}^{N} X_{i}(v,\varepsilon) y^{i} \right]^{j}.
\end{gathered}
\end{equation*}
Equating the homogeneous forms of $y$ in the left- and right-hand sides, we obtain
\begin{equation*}
\begin{gathered}
	C_{0}(v,\varepsilon) := G_{0}(v,\varepsilon),\\
	C_{1}(v,\varepsilon) y + U_{1}(v,\varepsilon) J(v,\varepsilon) y + \varepsilon \frac{\partial U_{1}(v,\varepsilon)y}{\partial v} C_{0}(v,\varepsilon) = \\
	= \varepsilon \frac{\partial G_{0}(v,\varepsilon)}{\partial v} \left[ U_{1}(v,\varepsilon) y \right] + G_{1}(v,\varepsilon)y,
\end{gathered}
\end{equation*}
\begin{equation*}
\begin{gathered}
	H_{k}(v,\varepsilon) y^{k} + \frac{\partial \left( X_{k}(v,\varepsilon) y^{k} \right)}{\partial y}J(v,\varepsilon) y + \varepsilon \frac{\partial \left( X_{k}(v,\varepsilon) y^{k} \right)}{\partial v} C_{0}(v,\varepsilon) = \\
 	= J(v,\varepsilon) X_{k}(v,\varepsilon) y^{k} + F_{k}(v,\varepsilon) y^{k} + M_{k}(v,\varepsilon) y^{k}, \quad k=2,\ldots,N,
 \end{gathered}
\end{equation*}
\begin{equation*}
\begin{gathered}
	C_{k}(y,\varepsilon) y^{k} + \frac{\partial [U_{k}(v,\varepsilon) y^{k}]}{\partial y} J(v,\varepsilon)y  + \varepsilon \frac{\partial \left( U_{k}(v,\varepsilon) y^{k} \right) }{\partial v} C_{0}(v,\varepsilon) = \\
	= \varepsilon \frac{\partial G_{0}(v,\varepsilon)}{\partial v}U_{k}(v,\varepsilon) y^{k} + G_{1}(v,\varepsilon) X_{k}(v,\varepsilon) y^{k} + G_{k}(v,\varepsilon) y^{k} + N_{k}(v,\varepsilon) y^{k}, \quad k=2,\ldots,N,
\end{gathered}
\end{equation*}
where all of the forms $M_{k}(u,\varepsilon)y^{k}$ and $N_{k}(u,\varepsilon)$ are determined by the forms contained in the original system, the resulting one and in the transformations and have indices less than $k$. The forms within the given equalities can be expanded as $H_{k}(u,\varepsilon) \sim \sum_{j \ge 0} \varepsilon^{j} H_{k,j}(u)$ and so on. Having introduced operators
\begin{equation*}
	\mathfrak{L}_{J_{0}(v)y} \, \cdot := \frac{\partial \, \cdot}{\partial y} J_{0}(y) y - J_{0}(v) \, \cdot, \quad
	\partial_{J_{0}(v)y} \, \cdot := \frac{\partial \, \cdot}{\partial y} J_{0}(y)y
\end{equation*}
and having equated coefficients in left- and right-hand sides, we arrive at the homological equations for determination of $X_{k,j}(u)$, $H_{k,j}(u)$, $U_{k,j}(u)$, $C_{k,j}(u)$:
\begin{equation*}
\begin{gathered}
	U_{1,j}(v) J_{0}(v) y = G_{1,j}(v) y - \sum_{i=0}^{j-1} U_{1,i}(v) J_{j - i}(v) y + \\
	+ \sum_{i=0}^{j-1} \left[ \frac{\partial C_{0,j-i-1}(v)}{\partial v} U_{1,i}(v) y - \frac{\partial U_{1,i}(v)y}{\partial v} C_{0,j-i-1}(v) \right] - C_{1,j}(v) y,
\end{gathered}
\end{equation*}
where $j=0,\ldots,s-1$, and
\begin{equation*}
\begin{gathered}
	\mathfrak{L}_{J_{0}(v)y} X_{k,j}(v) y^{k} = F_{k,j}(v) y^{k} + P_{k,j}(v) y^{k} - H_{k,j}(v) y^{k},\\
	\partial_{J_{0}(v)y} U_{k,j}(v) y^{k} = G_{1,0}(v) X_{k,j}(v) y^{k} + G_{k,j}(v) y^{k} + Q_{k,j}(v) y^{k} - C_{k,j}(v) y^{k},
\end{gathered}
\end{equation*}
where $k=2,\ldots,N, \; j =0,\ldots,s$. Here the forms $P_{k,j}(u)$, $Q_{k,j}(u)$ are constructed using the forms found from the analogous equations on the previous step, i.~e. from the equations in which there is an index $k-1$ instead of $k$. The forms $H_{k,j}(v)$ and $C_{k,j}(v)$ are chosen in such a way that the resulting system has, in some sense, as a simple structure as possible. Since $J_{0}(v)$ is non-degenerate, $U_{1,j}(v)$ is explicitly found by assigning $C_{1,j}(v)=0$, $j\ge0$. After that, we switch over to determining other required forms by replacing them with their expansions in the basis forms~\eqref{eq:bwp_basform}:
\begin{equation*}
	X_{k,j}(v) = \sum_{\left| \mathbf{q} \right| = k} \varsigma_{\mathbf{q}}(y) X_{\mathbf{q},j}(v) = \sum_{i=1}^{2n} \sum_{\left| \mathbf{q} \right| = k} X_{i,\mathbf{q},j}(v) e_{i,\mathbf{q}}(y), \quad
	U_{k,j}(v) = \sum_{\left| \mathbf{q} \right| = k} \varsigma_{\mathbf{q}}(y) U_{\mathbf{q},j}(v), \quad \text{etc.}
\end{equation*}
As $S^{-1}J_{0}(v)S := \mathrm{diag} \left[ \mathrm{i}\omega_{1}(v), \ldots, \mathrm{i}\omega_{n}(v), -\mathrm{i}\omega_{1}(v), \ldots, -\mathrm{i}\omega_{n}(v) \right]$, it is not hard to deduce the equalities
\begin{equation*}
	\partial_{J_{0}(v)y} \varsigma_{\mathbf{q}}(y) = \left[ \varsigma_{\mathbf{q}}(y) \right]_{y}^{\prime} J_{0}(v) y = \frac{\mathrm{d}}{\mathrm{d}t} \Bigl|_{t=0} \left[ \left(S^{-1}\mathrm{e}^{J_{0}(v)t} y \right)^{\mathbf{q}} \right] =\mathrm{i} \left\langle \omega(v), I\mathbf{q} \right\rangle \varsigma_{\mathbf{q}}(y),
\end{equation*}
\begin{equation*}
\begin{gathered}
	\mathfrak{L}_{J_{0}(v)y} e_{i,\mathbf{q}}(y) = \left[ e_{i,\mathbf{q}}(y) \right]_{y}^{\prime} J_{0}(v) y - J_{0}(v) e_{i,\mathbf{q}}(y) = \frac{\mathrm{d}}{\mathrm{d}t} \Bigl|_{t=0} \mathrm{e}^{-J_{0}(v)t} e_{i,\mathbf{q}} \left(\mathrm{e}^{J_{0}(v)t} y \right) = \\
	=\mathrm{i} \left\langle \omega(v), I(\mathbf{q} - \mathbf{e}_{i} \right\rangle e_{i,\mathbf{q}}(y).
\end{gathered}
\end{equation*}
But then the equations for determining the desired forms' coefficients become
\begin{equation*}
\begin{gathered}
\mathrm{i}\left\langle \omega(v), I(\mathbf{q} - \mathbf{e}_{i}) \right\rangle X_{i,\mathbf{q},j}(v) = F_{i,\mathbf{q},j}(v) - P_{i,\mathbf{q},j}(v) - H_{i,\mathbf{q},j}(v),\\
	\mathrm{i}\left\langle \omega(v), I\mathbf{q} \right\rangle U_{\mathbf{q},j}(v) = G_{1,0}(v) X_{\mathbf{q},j}(v) + G_{\mathbf{q},j}(v) + Q_{\mathbf{q},j}(v) - C_{\mathbf{q},j}(v).
\end{gathered}
\end{equation*}
If $v \in \mathcal{A}_{i}(N,\nu)$, $i \ne 0$, then in case $I(\mathbf{q} - \mathbf{e}_{i}) = 0$, we can declare $H_{i,\mathbf{q},j}(v) = F_{i,\mathbf{q},j}(v) - P_{i,\mathbf{q},j}(v)$, $X_{i,\mathbf{q},j}(v)=0$, otherwise $H_{i,\mathbf{q},j}(v) = 0$, and at the same time we will find $X_{i,\mathbf{q},j}(v)$. Similarly, if $v \in \mathcal{A}_{0}(N,\nu)$, then $C_{\mathbf{q},j}(v) = G_{1,0}(v) X_{\mathbf{q},j}(v) + G_{\mathbf{q},j}(v) + Q_{\mathbf{q},j}(v)$; $U_{\mathbf{q},j}(v) = 0$ if $I\mathbf{q} = 0$, and $C_{\mathbf{q},j}(v) = 0$ if $I\mathbf{q} \ne 0$, in which case we explicitly find $U_{\mathbf{q},j}(v)$.

Having performed the constructed change of variables in~\eqref{eq:smb_fs}, we obtain system~\eqref{eq:bwp_nf1}.
\end{proof}

Let us now switch over to complex variables $z=(z_{1}, \ldots, z_{n}) \in \mathbb{C}^{n}$ in system~\eqref{eq:bwp_nf1} with the substitution
\begin{equation*}
	y = \sum_{j=1}^{n} z_{j}s_{j} + \sum_{j=1}^{n} \bar{z}_{j} \bar{s}_{j} = 2\mathrm{Re} \left[ \sum_{j=1}^{n} z_{j}s_{j} \right].
\end{equation*}
System~\eqref{eq:bwp_nf1} takes the form
\begin{equation*}
\begin{gathered}
	\dot{z}_{j} = \left[ \varepsilon \bar{\alpha}_{j}(v,\varepsilon) + \mathrm{i} \bar{\omega}_{j}(v,\varepsilon) + \sum_{3 \le 2 \left| \mathbf{p} \right| + 1 \le N} h_{j,\mathbf{p}}(v,\varepsilon)( \left| z \right|)^{2 \mathbf{p}} \right] z_{j} + \\
	+ O \left( \left\Vert z \right\Vert ^{N+1} + \varepsilon^{s+1} \left\Vert z \right\Vert \right), \quad j=1,\ldots,n,\\
	\dot{v} = \varepsilon \left[ \sum_{0 \le 2 \left| \mathbf{p} \right| \le N} c_{\mathbf{p}}(v,\varepsilon)(\left| z \right|)^{2\mathbf{p}} + O \left( \left\Vert z \right\Vert^{N+1} + \varepsilon^{s} \right) \right],
\end{gathered}
\end{equation*}
where $\mathbf{p} \in \mathbb{Z}_{+}^{n}$ and $h_{j,\mathbf{p}}(v,\varepsilon) := H_{j,(\mathbf{p},\mathbf{p}) + \mathbf{e}_{j}}(v,\varepsilon)$, $c_{\mathbf{p}}(v,\varepsilon) := C_{(\mathbf{p},\mathbf{p})}(v,\varepsilon)$, $(\mathbf{p}, \mathbf{p}) := (p_{1}, \ldots, p_{n}, p_{1}, \ldots, p_{n})$, $(\left| z \right|) = \left( \left| z_{1} \right|,\ldots, \left| z_{n} \right| \right)$. The order relations $O \left( \left\Vert y \right\Vert ^{N+1} + \varepsilon^{s+1} \left\Vert y \right\Vert \right)$ and $O \left( \left\Vert y \right\Vert^{N+1} + \varepsilon^{s} \right)$ denote the remainder terms of the same kind as $\tilde{H}_{N,s+1}(y,v,\varepsilon)y$ and $\tilde{C}_{N+1,s}(y,v,\varepsilon)$ respectively. One can also easily ensure that the equations for $\bar{z}_{j}$ are complex conjugate with the equations for $z_{j}$.

Throughout the rest of this paper we will assume that the parameter $\varepsilon$ is non-negative. Having introduced the polar-type coordinates $r_{j}$, $\varphi_{j}| \bmod 2\pi$ by $z_{j}=\sqrt{\varepsilon  r_{j}} \mathrm{e}^{\mathrm{i} \varphi_{j}}$, $j=1,\ldots,n$, having defined $a_{j,\mathbf{p}}(v,\varepsilon) := \mathrm{Re} \, h_{j,\mathbf{p}}(v,\varepsilon)$, $b_{j,\mathbf{p}}(v,\varepsilon) := \mathrm{Im} \, h_{j,\mathbf{p}}(v,\varepsilon)$, $r = (r_{1}, \ldots, r_{n})$, $\sqrt{r} = (\sqrt{r_{1}}, \ldots, \sqrt{r_{n}})$, $\varphi = (\varphi_{1}, \ldots, \varphi_{n})$ and having assigned $s=(N+1)/2$, we come to the system
\begin{equation}\label{eq:polarnormform}
\begin{aligned}
	\dot{r}_{j} &= 2\varepsilon \left[ \bar{\alpha}_{j}(v,\varepsilon) + \sum_{3\le2\left|\mathbf{p}\right|+1\le N} \varepsilon^{\left|\mathbf{p}\right|-1} a_{j,\mathbf{p}}(v,\varepsilon) r^{\mathbf{p}} \right] r_{j} + \varepsilon^{N/2} \sqrt{r_{j}} R_{j}(r,v,\varphi,\varepsilon),\\
	\dot{v} &= \varepsilon \left[ \sum_{0\le2\left|\mathbf{p}\right|\le N} \varepsilon^{\left|\mathbf{p}\right|} c_{\mathbf{p}}(v,\varepsilon) r^{\mathbf{p}} + \varepsilon^{(N+1)/2} Z(r,v,\varphi,\varepsilon)\right],\\
	\dot{\varphi}_{j} &= \bar{\omega}_{j}(v,\varepsilon) + \sum_{3\le2\left|\mathbf{p}\right|+1\le N} \varepsilon^{\left|\mathbf{p}\right|} b_{j,\mathbf{p}}(v,\varepsilon)r^{\mathbf{p}} + \varepsilon^{N/2} r_{j}^{-1/2} \Phi_{j}(r,v,\varphi,\varepsilon), \quad j=1,\ldots,n,
\end{aligned}
\end{equation}
where the remainder terms can be written as
\begin{equation*}\label{m0}
\begin{gathered}
	R_{j}(r,v,\varphi,\varepsilon) := \sum_{\left| \mathbf{q} \right| = N+1} \tilde{a}_{j,\mathbf{q}}(\sqrt{\varepsilon r},v,\varphi,\varepsilon) \sqrt{r}^{\mathbf{q}} + 2\sum_{\left|\mathbf{q}\right|=1} \hat{a}_{j,\mathbf{q}}(\sqrt{\varepsilon r},v,\varphi,\varepsilon) \sqrt{r}^{\mathbf{q}},\\
	\Phi_{j}(r,v,\varphi,\varepsilon) := \sum_{\left|\mathbf{q}\right|=N+1} \tilde{b}_{j,\mathbf{q}}(\sqrt{\varepsilon r},v,\varphi,\varepsilon) \sqrt{r}^{\mathbf{q}} + \sum_{\left|\mathbf{q}\right|=1} \hat{b}_{j,\mathbf{q}}(\sqrt{\varepsilon r},v,\varphi,\varepsilon) \sqrt{r}^{\mathbf{q}},\\
	Z(r,v,\varphi,\varepsilon) := \sum_{\left|\mathbf{q}\right|=N+1} \tilde{c}_{\mathbf{q}}(\sqrt{\varepsilon r},v,\varphi,\varepsilon) \sqrt{r}^{\mathbf{q}} + \hat{c}(\sqrt{\varepsilon r},v,\varphi,\varepsilon).
\end{gathered}
\end{equation*}
Here the functions $\tilde{a}_{j,\mathbf{q}}(\rho,v,\varphi,\varepsilon), \hat{a}_{j,\mathbf{q}}(\rho,v,\varphi,\varepsilon),\ldots$ are smooth in $\left[ 0, \varrho_{0} \right]^{n} \times \mathbb{R}^{m} \times \mathbb{T}^{n} \times [0, \varepsilon_{0}]$, and $\varrho_{0} > 0$, $\varepsilon_{0} \ll 1$ are some positive constants.

If we declare for a couple of vectors $p=(p_{1}, \ldots, p_{n})$, $q=(q_{1}, \ldots, q_{n})$ an operation $p\bullet q = \left(p_{1}q_{1}, \ldots, p_{n}q_{n}\right)$ and assign
\begin{equation*}
	\alpha(v) := (\alpha_{1}(v), \ldots, \alpha_{n}(v)), \quad
	A(v ):= - \left\{  a_{i,\boldsymbol{\epsilon}_{j}}(v,0) \right\}_{i,j=1}^{n}, \quad
	c(v) := c_{0}(v,0)
\end{equation*}
(here $\boldsymbol{\epsilon}_{j}$ denotes the $j$-th coordinate unit vector of the space $\mathbb{R}^{n}$), it will allow us to rewrite system~\eqref{eq:polarnormform} as
\begin{equation}\label{eq:polnf_vec}
\begin{aligned}
	\dot{r} &= 2\varepsilon \left[ \alpha(v) - A(v) r + \varepsilon B(r,v,\varepsilon) \right] \bullet r + \varepsilon^{N/2} \sqrt{r} \bullet R(r,v,\varphi,\varepsilon),\\
	\dot{v} &= \varepsilon c(v) + \varepsilon^{2} W(r,v,\varepsilon) + \varepsilon^{(N+3)/2} Z(r,v,\varphi,\varepsilon),\\
	\dot{\varphi} &= \omega(v) + \varepsilon \Psi(r,v,\varepsilon) + \varepsilon^{N/2} r^{-1/2} \bullet \Phi(r,v,\varphi,\varepsilon).
\end{aligned}
\end{equation}
All of the functions which appear in this system are bounded in $\left[ 0, \varrho \right] ^{n} \times \mathbb{R}^{m} \times \mathbb{T}^{n} \times [0, \varepsilon_{0}]$, where $0 < \varrho < \varrho_{0} / \varepsilon_{0}$, and their smoothness properties are determined by the corresponding terms of system~\eqref{eq:polarnormform}. Furthermore, the order relations $\left\Vert R(r,v,\varphi,\varepsilon) \right\Vert = O (\left\Vert \sqrt{r} \right\Vert )$, $\left\Vert \Phi(r,v,\varphi,\varepsilon) \right\Vert = O (\left\Vert \sqrt{r} \right\Vert )$ hold uniformly in $v\in B_{R^{\ast}}^{m}$, $\varphi\in\mathbb{T}^{n}$, $\varepsilon\in[0,\varepsilon_{0}]$ when $\left\Vert r\right\Vert \to0$.

For the sake of simplicity, we shall consider the case when the domains $\mathcal{D}_{s}$, $\mathcal{D}_{\ast}$ and $\mathcal{D}_{u}$, which were mentioned in the Introduction, are formed by nested balls. More precisely, let us introduce the following notations for a triplet of numbers $R_{0}, R_{\ast}, R^{\ast}$ such that $0 < R_{0} < R_{\ast} < R^{\ast}$:
\begin{equation*}
\begin{gathered}
	\alpha_{0} := \min_{1\le j\le m} \inf_{v\in B_{R_{0}}^{m}} \alpha_{j}(v), \quad
	\alpha_{\ast} := -\max_{1\le j\le m} \sup_{v\in B_{R^{\ast}}^{m} \setminus B_{R_{\ast}}^{m}} \alpha_{j}(v), \quad
	\alpha^{\ast} := \sup_{v\in B_{R^{\ast}}^{m}} \left\Vert \alpha(v) \right\Vert,\\
	A_{\ast} := \inf_{v\in B_{R^{\ast}}^{m}} \min_{\left\Vert \xi\right\Vert =1} \left\langle A(v) \xi, \xi \right\rangle ,\quad
	A^{\ast} := \sup_{v\in B_{R^{\ast}}^{m}} \left\Vert A(v) \right\Vert
\end{gathered}
\end{equation*}
and state some additional assumptions.
\begin{description}
\item [{C4}] There exist such numbers $R_{0} < R_{\ast} < R^{\ast}$ that $\alpha_{\ast} > 0$, $\alpha_{0} > 0$, $A_{\ast} > 0$.
\item [{C5}] The conditions of resonances absence are fulfilled: $B_{R^{\ast}}^{m} \subset \mathcal{A}(N,\nu)$ for some $N\ge3$, $\nu>0$. Besides, if $N<5$, then $0 \in \mathcal{A}(5,\nu)$.
\item [{C6}] The conditions of system $\dot{v} = c(v)$ convergence are met: there exists $\varkappa>0$ such that $\left\langle  c(v), v \right\rangle < -\varkappa \left\Vert v \right\Vert ^{2}$ for all $v \in B_{R^{\ast}}^{m}$.
\item [{C7}] All components of vector $r^{\ast} := A^{-1}(0) \alpha(0)$  are positive.
\end{description}
Furthermore, without loss of generality we can assert that $r^{\ast} = (1,1,\ldots,1)$. This can always be achieved using the scaling transformation $r \mapsto r^{\ast} \bullet r$.

In accordance with condition~\textbf{C4}, we now take that
\begin{equation}\label{eq:bwp_D_sD_u}
	\mathcal{D}_{s} = B_{R^{\ast}}^{m} \setminus B_{R_{\ast}}^{m}, \quad
	\mathcal{D}_{\ast} = B_{R_{\ast}}^{m} \setminus B_{R_{0}}^{m}, \quad
	\mathcal{D}_{u} = B_{R_{0}}^{m}.
\end{equation}

\section{Analysis of the First Approximation System}\label{sec:first_approx}

In order to have at least rough understanding of how system~\eqref{eq:polnf_vec} solutions behave, let us focus on  the first approximation system
\begin{equation*}
	\dot{r} = 2\varepsilon[ \alpha(v) - A(v)r ] \bullet r, \quad
	\dot{v} = \varepsilon c(v), \quad
	\dot{\varphi} = \omega(v) + \varepsilon\Psi(r,v,\varepsilon).
\end{equation*}
Currently we are mainly interested in dynamics of the subsystem for the variables $r$, $v$
\begin{gather}
	\dot{r} = 2\varepsilon[ \alpha(v) - A(v)r] \bullet r,\label{eq:modsys_r}\\
	\dot{v} = \varepsilon c(v).\label{eq:modsys_v}
\end{gather}

\begin{proposition}\label{prop:arrt_v=00003D0}
The origin is a global attractor of system~\eqref{eq:modsys_v} within a ball $B_{R^{\ast}}^{m}$.
\end{proposition}

\begin{proof}
Condition \textbf{C6} implies that along every solution of system~\eqref{eq:modsys_v} the function $\left\langle v,v\right\rangle$ tends to zero monotonously.
\end{proof}

Let $\left\{  v(t)\right\} _{t\ge0}$ be a forward trajectory of system~\eqref{eq:modsys_v} and let us consider the solution $r(\cdot)$ of the system
\begin{equation*}
	\dot{r} = 2\varepsilon[ \alpha(v(t)) - A(v(t))r] \bullet r,
\end{equation*}
such that $r(0) \in (0,\infty)^{n}$. Since the $j$-th component $r_{j}(\cdot)$ of this solution can be viewed as a non-trivial solution of a linear homogeneous equation with a continuous coefficient, it yields $r_{j}(t)>0$ for all $t\ge0$.

\begin{proposition}\label{prop:attr_triv_sol}
Suppose that $v(0) \in \mathcal{D}_{s}$ and $T^{\ast} = \sup  \left\{  t \ge 0 \colon v(t) \in \mathcal{D}_{s}\right\}$. Then $\left| r(t) \right| \le \left| r(0) \right| \mathrm{e}^{-2\varepsilon\alpha_{\ast}t}$ for all $t \in [0,T^{\ast}]$.
\end{proposition}

\begin{proof}
By condition~\textbf{C4} on the interval $[0,T^{\ast}]$ we get
\begin{equation*}
	\frac{\mathrm{d} \left| r(t) \right|}{\mathrm{d}t} \le 2\varepsilon \left[ -\alpha_{\ast} \left| r(t) \right| - \left\langle  A(v(t))r(t), r(t) \right\rangle \right] \le 2\varepsilon \left[ -\alpha_{\ast} \left| r(t) \right| - A_{\ast} \left\Vert r(t) \right\Vert ^{2} \right ] \le - 2\varepsilon \alpha_{\ast} \left| r(t) \right|,
\end{equation*}
which provides the estimate for $\left| r(t) \right|$.
\end{proof}

\begin{corollary}\label{cor:deriv_|r|_1}
For all $v \in \mathcal{D}_{s}$ the derivative of $\left| r \right|$ along trajectories of subsystem~\eqref{eq:modsys_r} does not exceed $-2\varepsilon\alpha_{\ast} \left| r \right|$.
\end{corollary}

\begin{proposition}\label{prop:attrfixedpoint}
The point $\left( r^{\ast},0 \right)$ is a global attractor of system~\eqref{eq:modsys_r}--\eqref{eq:modsys_v} in the domain $(0,\infty)^{n} \times B_{R^{\ast}}^{m}$.
\end{proposition}

\begin{proof}
It follows from the inequality
\begin{equation*}
	\frac{\mathrm{d} \left| r(t) \right|}{\mathrm{d}t} \le 2\varepsilon \left[ \left| \alpha(v(t)) \bullet r(t) \right| - \left\langle A(v(t))r(t),r(t) \right\rangle \right] \le 2\varepsilon \left\Vert r(t) \right\Vert \left[ \alpha^{\ast} - A_{\ast} \left\Vert r(t) \right\Vert \right]
\end{equation*}
that $\left| r(t )\right|$ is decreasing while $\left\Vert r(t) \right\Vert > \alpha^{\ast} / A_{\ast}$, and hence, it does so for at least as long as all points of the hyperplane $\left| r \right|=\left|r(t)\right|$ stay outside the sphere $\left\Vert r \right\Vert  = \alpha^{\ast}/A_{\ast}$. Or in other words, it is decreasing while the distance from the hyperplane $\left|r\right| = \left|r(t)\right|$ to the origin is greater than $\alpha^{\ast} / A_{\ast}$. Since this distance equals $\left|r(t)\right| / \sqrt{n}$, then no matter how small $\delta>0$ we choose, there will be a unique positive moment of time starting from which $r(t)$ belongs to the bounded set $\left\{  r \in (0,\infty)^{n} \colon \left|r\right| \le \sqrt{n} \left( \alpha^{\ast} + \delta \right) / A_{\ast} \right\}$.

Meanwhile, as soon as at some moment $t_{0} \ge 0$ the point $v(t)$ enters $\mathcal{D}_{u}$, the inequality
\begin{equation*}
	\frac{\mathrm{d}\left|r(t)\right|}{\mathrm{d}t} \ge 2\varepsilon \left[ \alpha_{\ast} \left| r(t) \right| - A^{\ast} \left\Vert r(t) \right\Vert ^{2} \right] \ge 2\varepsilon \left| r(t) \right| \left[\alpha_{\ast} - A^{\ast} \left| r(t) \right|\right].
\end{equation*}
will become valid. Consequently, starting from some moment of time $t_{1}\ge t_{0}$ the inequality $\left| r(t) \right| \ge \left( \alpha_{\ast} - \delta\right) / A^{\ast}$ holds.  Thus, if we declare
\begin{equation}\label{eq:compK}
	\mathcal{K} :=  \left\{  r \in \mathbb{R}_{+}^{n} \colon (\alpha_{\ast} - \delta) / A^{\ast} \le \left| r \right| \le \sqrt{n} \left( \alpha^{\ast} + \delta \right) / A_{\ast} \right\},
\end{equation}
there will be a moment of time $t_{\mathcal{K}} = t_{\mathcal{K}}(r(0))\ge t_{1}$ such that $r(t_{\mathcal{K}}) \in \mathcal{K}$, and then $r(t) \in \mathcal{K}$ for every $t\ge t_{\mathcal{K}}$. It is noticeable that by choosing the  number $\varepsilon_{0}$ small enough, without loss of generality, we can thereby suppose that $\varrho > \sqrt{n} \left( \alpha^{\ast} + \delta \right) / A_{\ast}$, and accordingly, $\left|r\right| < \varrho$ for all $r \in \mathcal{K}$.

Let us prove that $r(t) \to r^{\ast}$ when $t \to+\infty$. Consider the limit system
\begin{equation*}
	\dot{r} = 2\varepsilon[ \alpha(0) - A(0) r] \bullet r.
\end{equation*}
It has a positive definite Lyapunov function in $(0,\infty)^{n}$ (relative to the equilibrium $r^{\ast}$)
\begin{equation}\label{eq:bwp_V_0}
	V_{0}(r) := \sum_{i=1}^{n} \left(r_{i} - 1 - \ln r_{i} \right) \equiv \left| r \right| - \ln \prod_{i}^{n} r_{i} - n,
\end{equation}
with a negative definite derivative along the limit system. Indeed,
\begin{equation*}
\begin{gathered}
	\left\langle \nabla V_{0}(r), 2\varepsilon[ \alpha(0) - A(0)r] \bullet r\right\rangle  = 2\varepsilon \sum_{i=1}^{n} \left( \frac{r_{i}-1}{r_{i}} \right) \left[ \alpha(0) - A(0) r \right]_{i} r_{i} = \\
	= 2\varepsilon \left\langle r - r^{\ast}, \alpha(0) - A(0) r \right\rangle = - 2\varepsilon\left\langle r - r^{\ast}, A(0)\left[ r - r^{\ast} \right] \right\rangle \le - 2\varepsilon A_{\ast} \left\Vert r - r^{\ast} \right\Vert^{2}.
\end{gathered}
\end{equation*}

Now let us assign $q:=\sup_{0 < \left\Vert v \right\Vert \le R^{\ast}} \left\Vert v \right\Vert^{-1} \left[ \left\Vert \alpha(v) - \alpha(0) \right\Vert + \left\Vert A(v) - A(0) \right\Vert \varrho\right]$ and compute the derivative along trajectories of system~\eqref{eq:modsys_r}--\eqref{eq:modsys_v} of the function
\begin{equation}\label{eq:bwp_def_V}
	V(r,v) := V_{0}(r) + \lambda \left\Vert v \right\Vert^{2} / 2,
\end{equation}
where $\lambda > q^{2} / \left(2A_{\ast}\varkappa\right)$. Sylvester's criterion of positive definiteness of a quadratic form claims the existence of such a number $\mu>0$ that
\begin{gather}
	\left\langle \nabla V_{0}(r), 2\varepsilon [\alpha(v) - A(v)r] \bullet r \right\rangle + \varepsilon \lambda \left\langle c(v), v \right\rangle = -2\varepsilon \left\langle r - r^{\ast}, A(0) \left[ r - r^{\ast} \right] \right\rangle + \nonumber \\
	+2\varepsilon\left\langle r-r^{\ast},\alpha(v(t))-\alpha(0)+\left[A(0)-A(v(t))\right]r\right\rangle +\varepsilon\lambda\left\langle c(v),v\right\rangle \le \nonumber \\
	\le-\varepsilon\left[2A_{\ast}\left\Vert r-r^{\ast}\right\Vert ^{2}-2q\left\Vert r-r^{\ast}\right\Vert \left\Vert v\right\Vert +\lambda\varkappa\left\Vert v\right\Vert ^{2}\right]\le-\varepsilon\mu\left[\left\Vert r-r^{\ast}\right\Vert ^{2}+\left\Vert v\right\Vert ^{2}\right]\label{eq:bwp_est_V+v}
\end{gather}
for all $v \in B_{R^{\ast}}^{m}$ and $r \in(0,\infty)^{n}$ such that $\left|r\right| \le \varrho$. Since $(r(t),v(t)) \in \mathcal{K}\times\mathcal{D}_{u}$ for all sufficiently large $t$, then $V(r(t),v(t))\to0$ for $t \to+\infty$. But thus $r(t)\to r^{\ast}$ when $t\to+\infty$.
\end{proof}

\begin{corollary}\label{cor:deriv_|r|_2}
For all $v\in B_{R^{\ast}}^{m}$ and $r\in(0,\infty)^{n}$ such that $\left|r\right| > \sqrt{n}(\alpha^{\ast} + \delta)/A_{\ast}$ the derivative of function $\left|r\right|$ along trajectories of subsystem~\eqref{eq:modsys_r}does not exceed $-2\varepsilon\delta \left( \alpha^{\ast} + \delta\right) / A_{\ast}$. If, otherwise, $v \in\mathcal{D}_{u}$ and $0 < \left|r\right| < (\alpha_{\ast} - \delta)/A^{\ast}$, this derivative is greater than $2\varepsilon\delta \left|r\right|$. The set $\mathcal{K} \times \mathcal{D}_{u}$ is a forward invariant set of system~\eqref{eq:modsys_r}--\eqref{eq:modsys_v}. Moreover, each forward trajectory of this system such that $(r(0),v(0))\in(0,\infty)^{n} \times B_{R^{\ast}}^{m}$ enters $\mathcal{K}\times\mathcal{D}_{u}$.
\end{corollary}

Let $J(r)$ denote the Jacobi matrix $\frac{\partial}{\partial r} \left[ \left( \alpha(0) - A(0) r \right) \bullet r \right]$ and
$H_{V_{0}}(r)$ be the Hesse matrix of the function $V_{0}(\cdot)$ at the point $r$. One can easily verify that the quadratic form $\left\langle H_{V_{0}}(r^{\ast})r,r\right\rangle $ is positive definite.

\begin{proposition}\label{prop:liearised_stab}
The linear system $\dot{r}=J(r^{\ast})r$ is asymptotically stable and the derivative of the quadratic form $\left\langle H_{V_{0}}(r^{\ast})r,r\right\rangle $ along trajectories of this system is negative definite.
\end{proposition}

\begin{proof}
We have the inequality
\begin{equation*}
\begin{gathered}
	-\left\langle r - r^{\ast},A(0) \left[ r - r^{\ast} \right] \right\rangle = \left\langle \nabla V_{0}(r), [\alpha(0)-A(0)r] \bullet r \right\rangle =\\
	= \left\langle H_{V_{0}}(r^{\ast}) (r-r^{\ast}), J(r^{\ast})(r-r^{\ast} \right\rangle + O(\left\Vert r-r^{\ast}\right\Vert ^{3}).
\end{gathered}
\end{equation*}
Since the left-hand side is a quadratic form, the right-hand side should also be one. This concludes that
\begin{equation*}
	\left\langle H_{V_{0}}(r^{\ast})r,J(r^{\ast})r\right\rangle =-\left\langle A(0)r,r\right\rangle \quad\forall r\in\mathbb{R}^{n}.
\end{equation*}
\end{proof}

\section{The Main Theorem}\label{sec:main_theorem}

Everywhere in what follows, we assume that $\varrho > \max \left\{  \sqrt{n}\left( \alpha^{\ast} + \delta\right) / A_{\ast}, 1 \right\}$, and therefore, the set $\mathcal{K}$ given by formula~\eqref{eq:compK} is contained in the simplex
\begin{equation*}
	\mathcal{S}_{\varrho} := \left\{  r\in\mathbb{R}_{+}^{n} \colon \left|r\right|\le\varrho\right\} \subset\left[0,\varrho\right]^{n}.
\end{equation*}
Before we proceed to the formulation of the main theorem, let us point out an important property of a union of sub-level sets of the function $V_{0}(\cdot)$, defined by~\eqref{eq:bwp_V_0}:
\begin{equation*}
	\bigcup_{c>0}V_{0}^{-1}([0,c])=(0,\infty)^{n}.
\end{equation*}
This fact, particularly, is a consequence of the following simple lemma.

\begin{lemma}\label{lem: bwp_Q_eps}
For $k > 0$ and $c \ge 0$ consider the set
\begin{equation*}
	\mathcal{Q}_{\varepsilon}(k,c) :=  \left\{  r\in\mathbb{R}^{n} \colon r_{j} \ge \mathrm{e}^{-c} \varepsilon^{k}, \; j=1, \ldots, n \right\}.
\end{equation*}
If $\varepsilon_{0} \in (0,1)$ and $c \ge 1$, then
\begin{equation*}
	V_{0}^{-1} \left( \left[ 0, \left| \ln\varepsilon^{k} \right| + c - 1 \right] \right) := \left\{ r \in (0, \infty)^{n} \colon V_{0}(r) \le \left| \ln\varepsilon^{k} \right| + c - 1 \right\} \subset \mathcal{Q}_{\varepsilon}(k,c) \quad \forall \varepsilon \in (0, \varepsilon_{0}).
\end{equation*}
If additionally $\left|\ln\varepsilon_{0}^{k/n}\right|>\varrho-\ln\varrho$, then
\begin{equation*}
	\mathcal{Q}_{\varepsilon} \left( \frac{k}{n}, 0 \right) \cap \mathcal{S_{\varrho}} \subset V_{0}^{-1} \left( \left[ 0, \left| \ln\varepsilon^{k} \right| \right] \right) \quad \forall \varepsilon \in (0, \varepsilon_{0}).
\end{equation*}
\end{lemma}

\begin{corollary}
The estimate $\mathrm{mes} \left( \mathcal{S}_{\varrho} \setminus V_{0}^{-1} \left( \left[ 0, \left| \ln\varepsilon^{k} \right|\right] \right) \right) = O(\varepsilon^{k/n})$ holds when $\varepsilon \to+0$.
\end{corollary}

Now we can proceed to the statement of our main result.

\begin{theorem}
Suppose that conditions~\textbf{C4}--\textbf{C7} are fulfilled and $0 < k < N - 2$. Then there exists such $\varepsilon_{0} > 0$ that: 1)~for every $\varepsilon \in (0, \varepsilon_{0})$ the solution $(r(t),v(t),\varphi(t))$ of system~\eqref{eq:polnf_vec} with the initial condition $(r(0), v(0), \varphi(0)) \in \mathcal{S}_{\varrho} \times \mathcal{D}_{s} \times \mathbb{T}^{n}$  can be extended to the semi-axis $[0, \infty)$, it satisfies the inequality $\left| r(t) \right| \le \left| r(0) \right| \mathrm{e}^{-\varepsilon\alpha_{\ast} t}$ on the interval $[0, T_{1}(\varepsilon)) :=  \left\{  t \ge 0 \colon v(t) \in \mathcal{D}_{s}\right\}$, and there exists such an instant $T_{2}(\varepsilon) > T_{1}(\varepsilon)$ that $r(t) \in \mathcal{K}$, $v(t) \in \mathcal{D}_{u}$ for $t \ge T_{2}(\varepsilon)$, where sets $\mathcal{K}$, $\mathcal{D}_{s}$ and $\mathcal{D}_{u}$ are defined by~\eqref{eq:compK},~\eqref{eq:bwp_D_sD_u}; 2)~system~\eqref{eq:polnf_vec} has an $n$-dimensional invariant torus $\mathcal{T}_{\varepsilon}$ located in an $O(\varepsilon)$-neighborhood of the torus $\left\{ r^{\ast} \right\} \times \left\{ 0\right\} \times \mathbb{T}^{n}$, and the system's restriction to $\mathcal{T}_{\varepsilon}$ has the form $\dot{\varphi} = \omega(0) + \varepsilon \mathfrak{f}(\varphi, \varepsilon)$ where $\mathfrak{f}(\cdot,\varepsilon) \colon \mathbb{T}^{n} \to \mathbb{R}^{n}$ is a Lipschitz vector field; 3)~if in addition $r(0)$ satisfies $V_{0}(r(0))\le\left|\ln\varepsilon^{k}\right|$, then there is a trajectory $\left\{ \left(\tilde{r}(t), \tilde{v}(t), \tilde{\varphi}(t) \right)\right\}_{t \in \mathbb{R}}$ on the torus $\mathcal{T}_{\varepsilon}$ such that
\begin{equation*}
	\lim_{t \to + \infty} \left[ \left| r(t) - \tilde{r}(t) \right| + \left| v(t) - \tilde{v}(t) \right| + \left| \varphi(t) - \tilde{\varphi}(t) \right| \right] = 0.
\end{equation*}
The statement remains valid for arbitrary $r(0) \in \mathcal{S}_{\varrho}$ if the additional conditions hold:
\begin{equation}\label{eq:bwp_R_j=00003D0}
	R_{j}(r,v,\varphi,\varepsilon) \bigl|_{r_{j}=0} = 0 \quad \forall (r, v, \varphi, \varepsilon) \in \mathcal{S}_{\varrho} \times B_{R^{\ast}}^{m} \times \mathbb{T}^{n} \times (0, \varepsilon_{0}), \; j=1, \ldots, n.
\end{equation}
\end{theorem}

The proof of this theorem is drawn from the statements of Sections~\ref{sec:pre_analysis} and~\ref{sec:existinvtor}.

\section{Preliminary Analysis of the Normalized System}\label{sec:pre_analysis}

In what follows assume that conditions of the main theorem are met. The following proposition captures a series of similarities between the dynamic of the first approximation system and the behavior of system~\eqref{eq:polnf_vec}.

\begin{proposition}\label{prop:glob_sol_nf}
Let $(r(t),v(t),\varphi(t))$, $t\in I$, be an non-extendible solution of system~\eqref{eq:polnf_vec} such that
$r(0) \in \mathcal{S}_{\varrho}$, $v(0) \in B_{R^{\ast}}^{m}$. Then for sufficiently small $\varepsilon_{0} > 0$ and every $\varepsilon \in (0, \varepsilon_{0})$ this solution has the following properties. 1) The interval $I$ contains the positive semi-axis, and hence, $\mathcal{S}_{\varrho} \times B_{R^{\ast}}^{m} \times \mathbb{T}^{n}$ is a forward invariant set of system~\eqref{eq:polnf_vec}. 2) There exists a moment of time $\tau_{\varepsilon} \ge 0$ after which $v(t)$ does not leave some $O(\varepsilon)$-neighborhood of the origin of $\mathbb{R}^{m}$. Moreover, on the interval $[0, \tau_{\varepsilon}]$ the function $\left\Vert v(t) \right\Vert $ is monotonously decreasing. 3) While  $v(t) \in \mathcal{D}_{s}$, the function $\left| r(t) \right|$ decreases with an exponential rate and it satisfies inequality $\left| r(t) \right| \le \left| r(0) \right| \mathrm{e}^{-\varepsilon\alpha_{\ast}t}$. 4) The set $\mathcal{K} \times \mathcal{D}_{u} \times \mathbb{T}^{n}$ is a forward invariant set of system~\eqref{eq:polnf_vec}, and there is a non-negative moment of time, since which $(r(t), v(t), \varphi(t)) \in \mathcal{K} \times \mathcal{D}_{u} \times \mathbb{T}^{n}$.
\end{proposition}

\begin{proof}
By computing and estimating using Corollaries~\ref{cor:deriv_|r|_1},~\ref{cor:deriv_|r|_2} the derivatives of the functions $\left| r \right|$ and $\left\langle v, v \right\rangle$ along trajectories of the corresponding subsystems of system~\eqref{eq:polnf_vec}, it is easy to verify that under the condition that $\varepsilon_{0}$ is sufficiently small in the corresponding domains these derivatives have the same signs as the derivatives of the functions $\left| r \right|$ and $\left\langle v, v \right\rangle$ along trajectories of system~\eqref{eq:modsys_r}--\eqref{eq:modsys_v}. In the same fashion as in proofs of Propositions~\ref{prop:arrt_v=00003D0},~\ref{prop:attr_triv_sol},~\ref{prop:attrfixedpoint}, we obtain the desired result.
\end{proof}

The presence of the term $\varepsilon^{N/2} \sqrt{r} \bullet R(r,v,\varphi,\varepsilon)$ in system~\eqref{eq:polnf_vec} makes it harder to establish a counterpart for Proposition~\ref{prop:attrfixedpoint}. The next proposition provides restrictions on initial values of system~\eqref{eq:polnf_vec} solution which guarantee that, starting from some moment of time, this solution enters and remains inside an $O(\sqrt{\varepsilon})$-neighborhood of the torus defined in the phase space by equations  $r=r^{\ast}$, $v=0$.

\begin{proposition}\label{prop:moment_t_eps}
There exist such positive numbers $\varepsilon_{0}$ and $C^{\ast}$ that for all $\varepsilon\in(0,\varepsilon_{0})$ and any solution $(r(t),v(t),\varphi(t))$ of system~\eqref{eq:polnf_vec} with the initial values $r(0) \in \mathcal{S}_{\varrho} \cap V_{0}^{-1} \left( [0,\left| \ln\varepsilon^{k} \right| ] \right)$ and $v(0) \in B_{R^{\ast}}^{m}$ there is a moment $t_{\varepsilon} > \tau_{\varepsilon}$ such that
\begin{equation}\label{eq:r-r*<}
	\left\Vert r(t) - r^{\ast} \right\Vert < C^{\ast} \sqrt{\varepsilon}, \quad \left\Vert v(t) \right\Vert  < C^{\ast} \varepsilon \quad \forall t > t_{\varepsilon}.
\end{equation}
If additionally conditions~\eqref{eq:bwp_R_j=00003D0} are met, then the existence of $t_{\varepsilon}$ is guaranteed for any solution of system~\eqref{eq:polnf_vec} such that $\left| r(0) \right| < \varrho$, $r_{j}(0) > 0$ $(j = 1, \ldots, n)$, $v(0) \in B_{R^{\ast}}^{m}$.
\end{proposition}

\begin{proof}
Turning back to Proposition~\ref{prop:glob_sol_nf}, it is enough to verify the first inequality~\eqref{eq:r-r*<}. Let $C_{0}$ be a constant that bounds from above each of norms $\left\Vert B(r, v, \varepsilon) \right\Vert$, $\left\Vert R(r, v, \varphi, \varepsilon) \right\Vert$, $\left\Vert W(r, v, \varepsilon) \right\Vert$, $\left\Vert Z(r, v, \varphi, \varepsilon) \right\Vert$ on the set $[0, \varrho]^{n} \times B_{R^{\ast}}^{m} \times \mathbb{T}^{n} \times [0, \varepsilon_{0}]$. By~\eqref{eq:bwp_est_V+v} on the set $(0,\varrho]^{n} \times \mathbb{T}^{n} \times  B_{R^{\ast}}^{m} \times [0, \varepsilon_{0}]$, the derivative of the function $V(r, v)$ (see~\eqref{eq:bwp_def_V}) along trajectories of system~\eqref{eq:polnf_vec} can be estimated as
\begin{equation*}
\begin{gathered}
	\dot{V}_{\eqref{eq:polnf_vec}}(r,v,\varphi,\varepsilon):=\left\langle \nabla V_{0}(r),2\varepsilon\left[\alpha(v)-A(v)r+\varepsilon B(r,v,\varepsilon)\right]\bullet r+\varepsilon^{N/2}\sqrt{r}\bullet R(r,v,\varphi,\varepsilon)\right\rangle +\\
	+\lambda\left\langle v,\varepsilon c(v)+\varepsilon^{2}W(r,v,\varepsilon)+\varepsilon^{(N+3)/2}Z(r,v,\varphi,\varepsilon)\right\rangle \le\\
	\le-\varepsilon\mu\left(\left\Vert r-r^{\ast}\right\Vert ^{2}+\left\Vert v\right\Vert ^{2}\right)+2\varepsilon^{2}\left\langle r-r^{\ast},B(r,v,\varepsilon)\right\rangle +\varepsilon^{N/2}\left\langle r-r^{\ast},r^{-1/2}\bullet R(r,v,\varphi,\varepsilon)\right\rangle +\\
	+\varepsilon^{2}\lambda\left\langle v,W(r,v,\varepsilon)\right\rangle +\varepsilon^{(N+3)/2}\lambda\left\langle v,Z(r,v,\varphi,\varepsilon)\right\rangle \le\\
	\le-\varepsilon\left[\mu\left(\left\Vert r-r^{\ast}\right\Vert ^{2}+\left\Vert v\right\Vert ^{2}\right)-\left\Vert r-r^{\ast}\right\Vert \left(2C_{0}\varepsilon+\varepsilon^{N/2-1}\left\Vert r^{-1/2}\bullet R(r,v,\varphi,\varepsilon)\right\Vert \right)-2\varepsilon\lambda C_{0}\left\Vert v\right\Vert \right].
\end{gathered}
\end{equation*}
Let us show that for sufficiently small $\varepsilon_{0}$ this derivative does not exceed a certain negative value on the set
\begin{equation}\label{eq:bwp_set_dotV<0}
	\left\{ (r, v, \varphi, \varepsilon) \in \left[\mathcal{Q}_{\varepsilon} \left(k,c\right) \cap \mathcal{S}_{\varrho} \right] \times B_{R^{\ast}}^{m}\times\mathbb{T}^{n} \times(0,\varepsilon_{0}] \colon \sqrt{\left\Vert r-r^{\ast}\right\Vert ^{2}+\left\Vert v\right\Vert ^{2}}\ge 6\sqrt{\varepsilon}C_{0}/\mu\right\},
\end{equation}
where $c \ge \lambda[R^{\ast}]^{2} / 2 + 1$. First of all, one can notice that
\begin{equation*}
	-\varepsilon \left\Vert v \right\Vert \left[ \mu \left\Vert v \right\Vert - 2 \varepsilon \lambda C_{0} \right] \le \varepsilon^{3} \lambda^{2} C_{0}^{2} / \mu \quad \forall v \in B_{R^{\ast}}^{m}.
\end{equation*}
Let $(r, v, \varphi, \varepsilon)$ belong to set~\eqref{eq:bwp_set_dotV<0}. If $\left\Vert r - r^{\ast} \right\Vert > 1/4$, then inequalities $r_{j} \ge \mathrm{e}^{-c}\varepsilon^{k}$ yield that for the first $n$ coordinates of the point of  set~\eqref{eq:bwp_set_dotV<0} we get
\begin{equation*}
	\varepsilon^{N/2 - 1} \left \Vert r^{-1/2} \bullet R(r, v, \varphi, \varepsilon) \right\Vert \le \varepsilon^{N/2 - 1} C_{0} \mathrm{e}^{c/2} \varepsilon^{-k/2} = \varepsilon^{(N-k-2)/2} C_{0} \mathrm{e}^{c/2},
\end{equation*}
and then for sufficiently small $\varepsilon_{0}$
\begin{equation*}
	\dot{V}_{\eqref{eq:polnf_vec}} (r, v, \varphi, \varepsilon) \le - \frac{\varepsilon}{4} \left[ \frac{\mu}{4} - 2C_{0} \varepsilon - \varepsilon^{(N-k-2)/2} C_{0} \mathrm{e}^{c/2}\right] + \varepsilon^{3}\lambda^{2 }C_{0}^{2} / \mu < 0\quad\forall\varepsilon\in(0,\varepsilon_{0}].
\end{equation*}
If, on the contrary, $\left\Vert r - r^{\ast}\right\Vert \le1/4$, then $\left|r_{j}-1\right|<1/4$, and thus, $r_{j}>1/4$. Taking into account that in this case $\left\Vert r^{-1/2}\bullet R(r, v, \varphi, \varepsilon) \right\Vert \le 2 C_{0}$,
for sufficiently small $\varepsilon_{0}$ and all $\varepsilon\in(0,\varepsilon_{0}]$ we obtain
\begin{equation*}
\begin{gathered}
	\dot{V}_{\eqref{eq:polnf_vec}} (r, v, \varphi, \varepsilon) \le - \varepsilon \left[ \mu \left( \left\Vert r - r^{\ast} \right\Vert^{2} + \left\Vert v \right\Vert^{2} \right) - 4 \sqrt{\varepsilon} C_{0} \left\Vert r - r^{\ast} \right\Vert  - 2 \varepsilon \lambda C_{0} \left\Vert v \right\Vert \right] \le\\
	\le - \varepsilon \left[ \mu \left( \left\Vert r - r^{\ast} \right\Vert^{2} + \left\Vert v \right\Vert^{2} \right) - 4 \sqrt{2\varepsilon} C_{0} \sqrt{\left\Vert r - r^{\ast} \right\Vert^{2} + \left\Vert v \right\Vert^{2}} \right] < 0.
\end{gathered}
\end{equation*}

In accordance with Lemma~\ref{lem: bwp_Q_eps} the set $V_{0}^{-1} \left( \left[ 0, \left| \ln\varepsilon^{k} \right| + c - 1 \right] \right)$ lies in $\mathcal{Q}_{\varepsilon}(k,c)$, which means that
\begin{equation*}
	\mathfrak{S} := V^{-1} \left( \left[ 0, \left| \ln\varepsilon^{k} \right| + c - 1 \right] \right) \cap \left[ \mathcal{S}_{\varrho} \times  B_{R^{\ast}}^{m} \right] \subset \left[ \mathcal{Q}_{\varepsilon}  \left( k, c \right) \cap \mathcal{S}_{\varrho} \right] \times B_{R^{\ast}}^{m},
\end{equation*}
and Proposition~\ref{prop:glob_sol_nf} together with the estimates for $\dot{V}_{\eqref{eq:polnf_vec}} (r, v, \varphi, \varepsilon)$ imply that the set $\mathfrak{S} \times \mathbb{T}^{n}$ is forward invariant. Furthermore, if  $(r(0), v(0)) \in \mathfrak{S}$, then there exists a moment $\tau_{\varepsilon}^{\ast} > 0$ such that $V(r(t), v(t)) < c^{\ast}(\varepsilon)$ for all $t > \tau_{\varepsilon}^{\ast}$, where
\begin{equation*}
	c^{\ast}(\varepsilon) = \max \left\{ V(r,v) \colon \sqrt{\left\Vert r - r^{\ast} \right\Vert^{2} + \left\Vert v \right\Vert^{2}} = 6 \sqrt{\varepsilon} C_{0} / \mu\right\}.
\end{equation*}
In fact, in a closed ball centered at $(r^{\ast},0)$, the function $V(\cdot, \cdot)$ reaches its maximal values only on the boundary. Therefore, in an open ball of radius $6\sqrt{2\varepsilon}C_{0} / \mu$, which is centered at $(r^{\ast}, 0)$, the function $V(\cdot,\cdot)$ takes values that are less than $c^{\ast}(\varepsilon)$, which means that this ball lies inside $V^{-1} \left( [0, c^{\ast} (\varepsilon)] \right)$. Now the existence of $t_{\varepsilon}$ results from the negativity of $\dot{V}(r,v,\varphi,\varepsilon)$ at points of set~\eqref{eq:bwp_set_dotV<0}.

Obviously, if $r \in \mathcal{S}_{\varrho}$, $V_{0}(r) \le \left| \ln\varepsilon^{k} \right|$ and $\left\Vert v \right\Vert \le R^{\ast}$, then $V(r,v) \le \left| \ln\varepsilon^{k} \right| + c - 1$. Hence, the set $\mathfrak{S}$ contains the set $\left[ V_{0}^{-1} \left( \left[ 0, \left| \ln\varepsilon^{k} \right|\right] \right) \cap \mathcal{S}_{\varrho} \right] \times B_{R^{\ast}}^{m}$.

Furthermore, since $V_{0}(r) \sim \frac{1}{2} \left \langle H_{V}(r^{\ast}) (r - r^{\ast}), r - r^{\ast} \right\rangle$ near $r^{\ast}$, one can specify such $C^{\ast} > 0$, that the set $V^{-1}([0,c^{\ast}(\varepsilon)])$ lies in a ball of radius $C^{\ast}\sqrt{\varepsilon}$ with center at $(r^{\ast},0)$ for all sufficiently small $\varepsilon > 0$. It enables us to assign $t_{\varepsilon} = \max \left\{  \tau_{\varepsilon}, \tau_{\varepsilon}^{\ast}\right\}$.

Finally, if conditions~\eqref{eq:bwp_R_j=00003D0} are met, then $R_{j}(r, v, \varphi, \varepsilon) = \sqrt{r_{j}} \tilde{R}_{j}(r, v, \varphi, \varepsilon)$, $j = 1, \ldots, n$ and the proof runs as before in case $r(0) \in \mathcal{S}_{\varrho} \cap (0, \infty)^{n}$ if subsystem for $r$~\eqref{eq:polnf_vec} is replaced by
\begin{equation*}
	\dot{r} = 2 \varepsilon \left[ \alpha(v) - A(v)r + \varepsilon B(r, v, \varepsilon) + \varepsilon^{N/2-1} \tilde{R}(r, v, \varphi, \varepsilon) \right] \bullet r.
\end{equation*}
\end{proof}

\begin{remark}
It follows from Propositions~\ref{prop:glob_sol_nf} and~\ref{prop:moment_t_eps} that when the forward trajectory $\bigcup_{t \ge 0}(r(t), v(t), \varphi(t))$ has no common points with the set $\mathfrak{S} \times \mathbb{T}^{n}$, then there is a moment of time after which $r(t) \in \mathcal{K} \setminus V_{0}^{-1} \left( [0, \left| \ln\varepsilon^{k} \right| ] \right)$.
\end{remark}

Hereafter, our main question will be whether system~\eqref{eq:polnf_vec} possesses an invariant torus close to the invariant torus of the first approximation system and if so, what its basin of attraction is.

\section{Existence of an Invariant Torus and its Basin of Attraction}\label{sec:existinvtor}

Taking into account the already proved propositions, we will conduct further analysis of system~\eqref{eq:polnf_vec} in a neighborhood of the torus $r=r^{\ast}$, $v=0$.

In order to simplify our notations, we will combine the variables $r$ and $v$ into one vector variable $y=(r,v)$ (this $(n+m)$-dimensional variable has no connection with the $2n$-dimensional local variable in Section~\ref{sec:const-normal-sys}). Let us rewrite system~\eqref{eq:polnf_vec} as
\begin{equation}\label{eq:sys_phi_y}
\begin{aligned}
	\dot{y}= & \varepsilon F(y,\varepsilon)+\varepsilon^{N/2}G(y,\varphi,\varepsilon),\\
	\dot{\varphi}= & \bar{\omega}(y,\varepsilon)+\varepsilon^{N/2}H(y,\varphi,\varepsilon),
\end{aligned}
\end{equation}
where
\begin{equation*}
\begin{gathered}
	F(y,\varepsilon):=\left(2\left[\alpha(v)-A(v)r+\varepsilon B(r,v,\varepsilon)\right]\bullet r,c(v)+\varepsilon W(r,v,\varepsilon)\right),\\
	G(y,\varphi,\varepsilon):=\left(\sqrt{r}\bullet R(r,v,\varphi,\varepsilon),\varepsilon^{3/2}Z(r,v,\varphi,\varepsilon)\right),\\
	\bar{\omega}(y,\varepsilon):=\omega(v)+\varepsilon\Psi(r,v,\varepsilon),\quad H(y,\varphi,\varepsilon):=r^{-1/2}\bullet\Phi(r,v,\varphi,\varepsilon),
\end{gathered}
\end{equation*}
and let us assign $y^{\ast} := (r^{\ast}, 0)$. Then $F(y^{\ast}, 0)=0$. Since
\begin{equation*}
	F^{\prime}(y^{\ast},0)= \begin{pmatrix}
		J(r^{\ast}) & 2\left[\alpha^{\prime}(0)-A^{\prime}(0)r^{\ast}\right]\bullet r^{\ast}\\
		0 & c^{\prime}(0)
	\end{pmatrix}
\end{equation*}
and according to Proposition~\ref{prop:liearised_stab} and condition~\textbf{C6} the linear systems $\dot{r} = J(r^{\ast})r$ and $\dot{v} = c^{\prime}(0)v$ are asymptotically stable, which means that all of the eigenvalues of these systems' matrices have negative real parts. Consequently, the system $\dot{y} = F^{\prime}(y^{\ast},0)y$ shares this property, too. It is well known that there exists a positive definite quadratic form which has a negative definite derivative along trajectories of an asymptotically stable linear system with a constant matrix. This positive definite quadratic form sets a dot product structure. Therefore, we will further assert that the space $\mathbb{R}^{n+m}$ is endowed with the dot product $\left\langle \cdot, \cdot \right\rangle$ for which the quadratic form $\left\langle F_{y}^{\prime}(y^{\ast},0)y, y \right\rangle $ is negative definite.

Now, we can choose positive numbers $\gamma$, $\sigma$ and $\varepsilon_{0}$ in such a way that the inequality
\begin{gather}
	\left\langle \left[F_{y}^{\prime}(y,\varepsilon)+\varepsilon^{(N-2)/2}G_{y}^{\prime}(y,\varphi,\varepsilon)\right]z,z\right\rangle \le-2\gamma\left\Vert z\right\Vert ^{2}\label{eq:est_dissip}\\
	\forall(y,z,\varphi,\varepsilon)\in B_{\sigma}^{n+m}(y^{\ast})\times\mathbb{R}^{n+m}\times\mathbb{T}^{n}\times[0,\varepsilon_{0}].\nonumber
\end{gather}
holds. It implies that $B_{\sigma}^{n+m}(y^{\ast}) \times \mathbb{T}^{n}$ is a forward invariant set of system~\eqref{eq:sys_phi_y}. Thus, for each point $(y, \varphi) \in B_{\sigma}^{n+m}(y^{\ast}) \times \mathbb{T}^{n}$ its forward trajectory, denoted by $ \left\{ (\eta_{t}(y, \varphi)),\phi_{t}(y, \varphi) \right\} _{t \ge 0}$, lies in $B_{\sigma}^{n+m}(y^{\ast}) \times \mathbb{T}^{n}$. In other words, in $B_{\sigma}^{n+m}(y^{\ast})\times\mathbb{T}^{n}$ system~\eqref{eq:sys_phi_y} generates the semi-flow
\begin{equation*}
	 \left\{ \left( \eta_{t} (\cdot, \cdot), \phi_{t} (\cdot, \cdot) \right) \colon B_{\sigma}^{n+m}(y^{\ast}) \times \mathbb{T}^{n} \to B_{\sigma}^{n+m} (y^{\ast}) \times \mathbb{T}^{n} \right\}_{t\ge0}.
\end{equation*}
It should be mentioned, that actually, as it follows from results of Section~\ref{sec:first_approx}, system~\eqref{eq:sys_phi_y} generates a semi-flow on the set $\mathcal{S}_{\varrho} \times  B_{R^{\ast}}^{m} \times \mathbb{T}^{n}$. Moreover, each point that enters the set $\left[ \mathcal{S}_{\varrho} \cap  V_{0}^{-1} \left( [0, \left| \ln\varepsilon^{k} \right| ] \right) \right] \times B_{R^{\ast}}^{m} \times \mathbb{T}^{n}$ at some moment of time under action of this semi-flow necessarily enters $B_{\sigma}^{n+m}(y^{\ast}) \times \mathbb{T}^{n}$ after some instant of time and later on keeps moving inside $O(\sqrt{\varepsilon})$-neighborhood of the torus $\{y^{\ast}\}\times\mathbb{T}^{n}$.

Now, let us show that for any sufficiently small $\varepsilon > 0$ the set $B_{\sigma}^{n+m}(y^{\ast}) \times \mathbb{T}^{n}$ contains an $n$-dimensional invariant torus of system~\eqref{eq:sys_phi_y} which attracts all forward trajectories of this set, and hence, the attraction basin of this torus contains the set $\left[ \mathcal{S}_{\varrho} \cap V_{0}^{-1} \left( [0, k \left| \ln\varepsilon \right| ] \right) \right] \times B_{R^{\ast}}^{m} \times \mathbb{T}^{n}$.

\begin{remark}\label{rem: N5}
If $\sigma$ is small enough, condition~\textbf{C5} yields that $B_{\sigma}^{m} \subset \mathcal{A} (5,\nu)$ and then for $v \in B_{\sigma}^{m}$ results of Section~\ref{sec:const-normal-sys} remain valid for $N=5$. Because of this, we will further consider system~\eqref{eq:sys_phi_y} with $N \ge 5$.
\end{remark}

It follow from~\eqref{eq:est_dissip} that the evolution matrix $\Omega_{s}^{t}$ of the linear system
\begin{equation*}\label{eq:linsysP}
	\dot{z} = \varepsilon \left[ F_{y}^{\prime} (\eta_{t} (y, \varphi),\varepsilon) + \varepsilon^{(N-2)/2} G_{y}^{\prime} (\eta_{t} (y, \varphi), \phi_{t} (y, \varphi), \varepsilon) \right] z =: \varepsilon P(t; y, \varphi, \varepsilon) z,
\end{equation*}
can be estimated as
\begin{equation*}
	\left\Vert \Omega_{s}^{t} \right\Vert \le\mathrm{e}^{-2\varepsilon\gamma\cdot(t-s)}, \quad t \ge s \ge0.
\end{equation*}

One can choose the positive numbers $\sigma$ and $K$ in such a way, that for all $(y^{i}, \varphi^{i}, \varepsilon) \in B_{2\sigma}^{n+m} (y^{\ast}) \times \mathbb{T}^{n} \times [0, \varepsilon_{0}]$ the following inequalities will hold
\begin{gather}
	\left\Vert \bar{\omega} \left( y^{1}, \varepsilon \right) + \varepsilon^{N/2} H \left(y^{1}, \varphi^{1}, \varepsilon \right) - \bar{\omega} \left(y^{2}, \varepsilon \right) - \varepsilon^{N/2} H \left(y^{2}, \varphi^{2}, \varepsilon \right) \right\Vert \le \nonumber \\
	\le K \left[ \left\Vert y^{1} - y^{2} \right\Vert + \varepsilon^{N/2} \left\Vert \varphi^{1} - \varphi^{2} \right\Vert \right], \label{eq:ineq_om+H}\\
	\left\Vert F(y^{1}, \varepsilon) + \varepsilon^{(N-2)/2} G(y^{1}, \varphi^{1}, \varepsilon) - F(y^{2}, \varepsilon) - \varepsilon^{(N-2)/2} G(y^{2}, \varphi^{2}, \varepsilon) - \right. \nonumber \\
	\left. - \left[ F_{y}^{\prime} (y^{3}, \varepsilon) + \varepsilon^{(N-2)/2} G_{y}^{\prime} (y^{3}, \varphi^{3}, \varepsilon) \right] (y^{1} - y^{2}) \right\Vert \le \nonumber \\
	\le \left\Vert F(y^{1}, \varepsilon) + \varepsilon^{(N-2)/2} G(y^{1}, \varphi^{2}, \varepsilon) - F(y^{2}, \varepsilon) - \varepsilon^{(N-2)/2} G(y^{2}, \varphi^{2}, \varepsilon) - \right. \nonumber \\
	\left. - \left[ F_{y}^{\prime} (y^{2}, \varepsilon) + \varepsilon^{(N-2)/2} G_{y}^{\prime} (y^{2}, \varphi^{2}, \varepsilon) \right] (y^{1} - y^{2}) \right\Vert + \nonumber \\
	+\left\Vert \left[ F_{y}^{\prime} (y^{2}, \varepsilon) - F_{y}^{\prime} (y^{3}, \varepsilon) + \varepsilon^{(N-2)/2} G_{y}^{\prime} (y^{2}, \varphi^{2}, \varepsilon) - \varepsilon^{(N-2)/2} G_{y}^{\prime}(y^{3}, \varphi^{3}, \varepsilon) \right] (y^{1} - y^{2}) \right\Vert + \nonumber \\
	+\varepsilon^{(N-2)/2} \left\Vert G(y^{1}, \varphi^{1}, \varepsilon) - G(y^{1}, \varphi^{2}, \varepsilon) \right\Vert \le \nonumber \\
	\le K \left[ \left( \left\Vert y^{1} - y^{2} \right\Vert + \left\Vert y^{2} - y^{3} \right\Vert  + \varepsilon^{(N-2)/2} \left\Vert \varphi^{2} - \varphi^{3} \right\Vert \right) \left\Vert y^{1} - y^{2} \right\Vert  + \varepsilon^{(N-2)/2} \left\Vert \varphi^{1} - \varphi^{2} \right\Vert \right].\label{eq:ineq_F_G}
\end{gather}

\begin{proposition}\label{prop:def_theta}
Let us declare $\mathcal{B}_{\varepsilon} := \left\{  (y,z) \in \mathbb{R}^{n+m} \times \mathbb{R}^{n+m} \colon y, z \in B_{\sigma}^{n+m}(y^{\ast}), \left\Vert y - z \right\Vert \le \varepsilon \right\}$ and $M := 4K / \gamma$. There exists such $\varepsilon_{0}>0$, that when $\varepsilon \in (0,\varepsilon_{0})$ to each point $\left(y,z,\varphi\right)\in\mathcal{B}_{\varepsilon}\times\mathbb{T}^{n}$ there is a unique corresponding point $\theta(y,z,\varphi)\in\mathbb{T}^{n}$ such that for all $t \ge 0$ the inequalities
\begin{equation}\label{eq:ineq_varphi-phi}
\begin{aligned}
	\left\Vert \eta_{t} \left(y, \varphi \right) - \eta_{t} \left( z, \theta(y, z, \varphi) \right) \right\Vert & \le 2\mathrm{e}^{-\varepsilon\gamma t} \left\Vert y - z \right\Vert,\\
	\left\Vert \phi_{t} \left(y, \varphi \right) - \phi_{t} \left(z, \theta(y, z, \varphi) \right) \right\Vert & \le \frac{M}{\varepsilon} \mathrm{e}^{-\varepsilon\gamma t} \left\Vert y - z \right\Vert.
\end{aligned}
\end{equation}
hold. Furthermore, $\theta(\cdot, \cdot, \cdot) \in \mathrm{C} \left( \mathcal{B}_{\varepsilon} \times \mathbb{T}^{n} \! \to \! \mathbb{T}^{n} \right)$; for every $\left( y, z^{i}, \varphi\right) \in \mathcal{B}_{\varepsilon} \times \mathbb{T}^{n}$, $i=1,2,$ it is true that
\begin{equation*}
	\left\Vert \theta \left( y, z^{1}, \varphi \right) - \theta \left( y, z^{2}, \varphi \right) \right\Vert \le\frac{M}{\varepsilon} \left\Vert z^{1} - z^{2} \right\Vert,
\end{equation*}
and for any fixed point $(y,z) \in \mathcal{B}_{\varepsilon}$ the mapping $\theta(y,z,\cdot) \colon \mathbb{T}^{n} \to \mathbb{T}^{n}$ is a homeomorphism.
\end{proposition}

\begin{proof}
Set $\mathfrak{M}_{\varepsilon}$ to be the space of continuous mappings
\begin{equation*}
	\mathcal{U}_{\varepsilon} := \mathbb{R}_{+} \times \mathcal{B}_{\varepsilon} \times \mathbb{T}^{n} \ni (t, y, z, \varphi) \mapsto (\zeta(t, y, z, \varphi), \psi(t, y, z, \varphi)) \in \mathbb{R}^{n+m} \times \mathbb{T}^{n}
\end{equation*}
which for all $(t,y,z,\varphi), \; (t,y,z^{i}, \varphi) \in \mathbb{R}_{+} \times \mathcal{B}_{\varepsilon} \times \mathbb{T}^{n}$, $i=1,2$ satisfy the inequalities
\begin{equation*}
\begin{gathered}
	\left\Vert \eta_{t}(y, \varphi) - \zeta(t, y, z, \varphi) \right\Vert \le 2\mathrm{e}^{-\varepsilon\gamma t} \left\Vert y - z \right\Vert,\\
	\left\Vert \phi_{t}(y, \varphi) - \psi(t, y, z, \varphi) \right\Vert \le \frac{M}{\varepsilon} \mathrm{e}^{-\varepsilon\gamma t} \left\Vert y - z \right\Vert,\\
	\left\Vert \zeta(t, y, z^{1}, \varphi) - \zeta(t, y, z^{2}, \varphi) \right\Vert \le 2\mathrm{e}^{-\varepsilon\gamma t} \left\Vert z^{1} - z^{2} \right\Vert,\\
	\left\Vert \psi(t, y, z^{1}, \varphi) - \psi(t, y, z^{2}, \varphi) \right\Vert \le \frac{M}{\varepsilon} \mathrm{e}^{-\varepsilon\gamma t} \left\Vert z^{1} - z^{2}\right\Vert
\end{gathered}
\end{equation*}
and the equality $\zeta(0,y,z,\varphi) = z$.

To make further notations shorter where it does not lead to confusion, for functions of $t,y,z,\varphi$ we will only indicate dependency on the time variable $t$ and write $\eta_{t}$, $\phi_{t}$, $\zeta_{t}$ and $\psi_{t}$ instead of $\eta_{t}(y,\varphi)$, $\phi_{t}(y,\varphi)$, $\zeta(t,y,z,\varphi)$ and $\psi(t,y,z,\varphi)$ respectively. Besides,  Remark~\ref{rem: N5} allows us to assume that $N=5$.

Let us introduce a structure of a complete metric space in $\mathfrak{M}_{\varepsilon}$ having defined the distance between a couple of elements $\left( \zeta^{i}, \psi^{i} \right) \in \mathfrak{M}_{\varepsilon}$, $i = 1, 2$ with the formula
\begin{equation*}
	d \left[ (\zeta^{1}, \psi^{1}), (\zeta^{2}, \psi^{2}) \right] := M \sup_{\mathcal{U}_{\varepsilon}} \left[ \mathrm{e}^{\varepsilon\gamma t} \left\Vert \zeta_{t}^{1} - \zeta_{t}^{2} \right\Vert \right] + \varepsilon \sup_{\mathcal{U}_{\varepsilon}} \left[ \mathrm{e}^{\varepsilon\gamma t} \left\Vert \psi_{t}^{1}-\psi_{t}^{2}\right\Vert \right].
\end{equation*}
On $\mathfrak{M}_{\varepsilon}$ we declare mappings
\begin{equation*}
\begin{gathered}
	\mathcal{F} [\zeta, \psi](t) := \phi_{t} + \intop_{t}^{\infty} \left[ \bar{\omega}(\eta_{s}, \varepsilon) - \bar{\omega}(\zeta_{s}, \varepsilon) \right] \mathrm{d} s + \varepsilon^{5/2} \intop_{t}^{\infty} \left[ H (\eta_{s}, \phi_{s}, \varepsilon) - H (\zeta_{s}, \psi_{s}, \varepsilon) \right] \mathrm{d} s,\\
	\mathcal{G} [\zeta, \psi](t) := \Omega_{0}^{t}z + \varepsilon\intop_{0}^{t} \Omega_{s}^{t} \left[ F (\zeta_{s}, \varepsilon) + \varepsilon^{3/2} G (\zeta_{s}, \psi_{s}, \varepsilon) - \left(F_{y}^{\prime}(\eta_{s},\varepsilon) + \varepsilon^{3/2} G_{y}^{\prime} (\eta_{s}, \phi_{s}, \varepsilon) \right) \zeta_{s} \right] \mathrm{d} s.
\end{gathered}
\end{equation*}
For a fixed bundle $\left( y, z, \varphi \right) \in \mathcal{B}_{\varepsilon} \times \mathbb{T}^{n}$ an element of the space $\mathfrak{M}_{\varepsilon}$ generates a solution of system~\eqref{eq:sys_phi_y} if and only if
\begin{equation}\label{eq:fixedpoint}
	\zeta_{t} = \mathcal{G} [\zeta, \psi](t), \quad \psi_{t} = \mathcal{F} [\zeta, \psi](t) \quad \forall t \ge 0.
\end{equation}
Indeed, the given element of the space $\mathfrak{M}_{\varepsilon}$ generates a solution if and only if for $t \ge 0$ the inequalities hold:
\begin{equation*}
\begin{gathered}
	\zeta_{t}=\mathcal{G}[\zeta,\psi](t),\\
	 \psi_{t}-\phi_{t}=\psi_{0}-\varphi+\intop_{0}^{t}\left[\bar{\omega}(\zeta_{s},\varepsilon)-\bar{\omega}(\eta_{s},\varepsilon)+\varepsilon^{5/2}\left(H(\zeta_{s},\psi_{s},\varepsilon)-H(\eta_{s},\phi_{s},\varepsilon)\right)\right]\mathrm{d}s.
\end{gathered}
\end{equation*}
The second one is obvious, whereas to derive the first one, it suffices to write down the solution $y = \zeta_{t}$ of the linear non-homogeneous system $\dot{y} = \varepsilon P(t;y,\varphi,\varepsilon)y + f(t)$, where
\begin{equation*}
	f(t) := \varepsilon \left[ F (\zeta_{t}, \varepsilon) + \varepsilon^{3/2} G(\zeta_{t}, \psi_{t}, \varepsilon) - \left( F_{y}^{\prime} (\eta_{t}, \varepsilon) + \varepsilon^{3/2} G_{y}^{\prime} (\eta_{t}, \phi_{t}, \varepsilon) \right) \zeta_{t} \right],
\end{equation*}
which takes the value $z$ at $t=0$. Since $\phi_{t} - \psi_{t} \to 0, \; t \to \infty$, we obtain the only possible initial value
\begin{equation}\label{eq:def_phi0}
	\psi_{0} = \varphi - \intop_{0}^{\infty} \left[ \bar{\omega}(\zeta_{s}, \varepsilon) - \bar{\omega}(\eta_{s}, \varepsilon) + \varepsilon^{5/2} \left( H (\zeta_{s}, \psi_{s}, \varepsilon) - H (\eta_{s}, \phi_{s}, \varepsilon) \right) \right] \mathrm{d} s,
\end{equation}
which leads to $\psi_{t} = \mathcal{F} [\zeta, \psi](t)$. Vice versa, if equalities~\eqref{eq:fixedpoint} are true, then it is evident that $t \mapsto (\zeta_{t}, \psi_{t})$ is a solution of system~\eqref{eq:sys_phi_y}.

Let us show that the choose suitable $\varepsilon_{0}$ for the mapping
\begin{equation*}
	\mathfrak{M}_{\varepsilon} \ni (\zeta, \psi) \mapsto \left( \mathcal{G}, \mathcal{F} \right)
\end{equation*}
to be a contraction. In the following we assume that $2 \varepsilon_{0} < \sigma$. Suppose that $(\zeta, \psi) \in \mathfrak{M}_{\varepsilon}$. Then inequality~\eqref{eq:ineq_om+H} yields
\begin{equation*}
\begin{gathered}
	\left\Vert \mathcal{F} [\zeta, \psi](t) - \phi_{t} \right\Vert \le K\intop_{t}^{\infty} \left[ \left\Vert  \zeta_{s} - \eta_{s} \right\Vert + \varepsilon^{5/2} \left\Vert \psi_{s} - \phi_{s} \right\Vert \right] \mathrm{d} s \le\\
	\le \frac{K \left[ 2 \left\Vert y - z \right\Vert + M \varepsilon^{3/2} \left\Vert y - z \right\Vert \right] \mathrm{e}^{-\varepsilon\gamma t}}{\varepsilon\gamma} \le \frac{M}{\varepsilon} \left[ \frac{1}{2} + \frac{M\varepsilon_{0}^{3/2}}{4} \right] \mathrm{e}^{-\varepsilon\gamma t} \left\Vert y - z \right\Vert \le \frac{M}{\varepsilon} \mathrm{e}^{-\varepsilon\gamma t} \left\Vert y - z \right\Vert
\end{gathered}
\end{equation*}
under condition that
\begin{equation}\label{eq:ineq1}
	\varepsilon_{0}^{3/2} \le 2 / M.
\end{equation}
If we take into account the fact that
\begin{equation*}
	\eta_{t} = \Omega_{0}^{t} y + \varepsilon \intop_{0}^{t} \Omega_{s}^{t} \left[ F (\eta_{s}, \varepsilon) + \varepsilon^{3/2} G (\eta_{s}, \phi_{s}, \varepsilon) - \left(F_{y}^{\prime} (\eta_{s}, \varepsilon) + \varepsilon^{3/2} G_{y}^{\prime} (\eta_{s}, \phi_{s}, \varepsilon) \right) \eta_{s} \right] \mathrm{d} s
\end{equation*}
and inequality~\eqref{eq:ineq_F_G} with $y^{1} = \zeta_{s}$, $\varphi^{1} = \psi_{s}$, $y^{2} = y^{3} = \eta_{s}$, $\varphi^{2} = \varphi^{3} = \phi_{s}$, it will give us
\begin{equation*}
\begin{gathered}
	\left\Vert \mathcal{G} [\zeta, \psi] (t) - \eta_{t} \right\Vert \le \mathrm{e}^{-2\varepsilon\gamma t} \left\Vert y - z \right\Vert +\\
	+ \varepsilon \intop_{0}^{t} \mathrm{e}^{-2\varepsilon\gamma\cdot(t-s)} K \left[ \left( \left\Vert  \zeta_{s} - \eta_{s} \right\Vert + \varepsilon^{3/2} \left\Vert \psi_{s} - \phi_{s} \right\Vert \right) \left\Vert  \zeta_{s} - \eta_{s} \right\Vert + \varepsilon^{3/2} \left\Vert \psi_{s} - \phi_{s} \right\Vert \right] \mathrm{d} s \le\\
	\le \mathrm{e}^{-2\varepsilon\gamma t} \left\Vert y - z \right\Vert + 4K \varepsilon^{2} \mathrm{e}^{-2\varepsilon\gamma t} t \left\Vert y - z \right\Vert + \frac{KM\varepsilon^{1/2}}{\gamma} \mathrm{e}^{-\gamma\varepsilon t} \left\Vert y - z \right\Vert \le\\
	\le \left[ 1 + \frac{M\varepsilon_{0}}{\mathrm{e}} + \frac{M^{2} \varepsilon_{0}^{1/2}}{4} \right] \mathrm{e}^{-\varepsilon\gamma t} \left\Vert y - z \right\Vert \le 2\mathrm{e}^{-\varepsilon\gamma t} \left\Vert y - z \right\Vert
\end{gathered}
\end{equation*}
as long as
\begin{equation*}
	\frac{M\varepsilon_{0}}{\mathrm{e}} + \frac{M^{2}\varepsilon_{0}^{1/2}}{4} \le 1.
\end{equation*}

Next, having set $\zeta_{t}^{i} := \zeta(t,y,z^{i}, \varphi)$, $\psi_{t}^{i} := \psi(t,y,z^{i},\varphi)$, $i = 1, 2$ and having made the assignments in inequalities~\eqref{eq:ineq_om+H}) and~\eqref{eq:ineq_F_G}) $(y^{i},\varphi^{i}) = (\zeta_{s}^{i}, \psi_{s}^{i})$, $i = 1, 2$, $y^{3} = \eta_{s}$, $\varphi^{3} = \phi_{s}$, we get
\begin{equation*}
\begin{gathered}
	\mathrm{e}^{\varepsilon\gamma t} \left[ \left\Vert \mathcal{F} [\zeta^{1}, \psi{}^{1}] (t) - \mathcal{F} [\zeta^{2}, \psi^{2}] (t) \right\Vert \right] \le\\
	\le \mathrm{e}^{\varepsilon\gamma t} \intop_{t}^{\infty} \mathrm{e}^{-\varepsilon\gamma s} K \mathrm{e}^{\varepsilon\gamma s} \left[ \left\Vert \zeta_{s}^{1} - \zeta_{s}^{2} \right\Vert  + \varepsilon^{5/2} \left\Vert \psi_{s}^{1} - \psi_{s}^{2} \right\Vert \right] \mathrm{d} s \le\\
	\le \frac{M}{4\varepsilon} \sup_{s\ge0} \left[ \mathrm{e}^{\varepsilon\gamma s} \left\Vert  \zeta_{s}^{1} - \zeta_{s}^{2} \right\Vert \right] + \frac{M\varepsilon_{0}^{3/2}}{4} \sup_{s\ge0} \left[ \mathrm{e}^{\varepsilon\gamma s} \left\Vert \psi_{s}^{1} - \psi_{s}^{2} \right\Vert \right] \le\\
	\le \left[ \frac{1}{2} + \frac{M\varepsilon_{0}^{3/2}}{4} \right] \frac{M}{\varepsilon} \left\Vert z^{1} - z^{2} \right\Vert,
\end{gathered}
\end{equation*}
\begin{equation*}
\begin{gathered}
	\mathrm{e}^{\varepsilon\gamma t} \left[ \left\Vert \mathcal{G} [\zeta^{1}, \psi{}^{1}] (t) - \mathcal{G} [\zeta^{2}, \psi^{2}] (t) \right\Vert \right] \le \mathrm{e}^{-\varepsilon\gamma t} \left\Vert z^{1} - z^{2} \right\Vert +\\
	+ \mathrm{e}^{-\varepsilon\gamma t} \varepsilon K \left[ \intop_{0}^{t} \mathrm{e}^{\varepsilon\gamma s} \left[ \left\Vert \zeta_{s}^{1} - \eta_{s} \right\Vert + 2 \left\Vert \zeta_{s}^{2} - \eta_{s} \right\Vert + \varepsilon^{3/2} \left\Vert \psi_{s}^{2} - \phi_{s} \right\Vert \right] \mathrm{d}  s \right] \sup_{s\ge0} \left[\mathrm{e}^{\varepsilon\gamma s} \left\Vert \zeta_{s}^{1} - \zeta_{s}^{2} \right\Vert \right] +\\
	+ \mathrm{e}^{-\varepsilon\gamma t} \varepsilon^{5/2} K \left[ \intop_{0}^{t} \mathrm{e}^{\varepsilon\gamma s} \mathrm{d} s \right] \sup_{s\ge0} \left[ \mathrm{e}^{\varepsilon\gamma s} \left\Vert  \psi_{s}^{1} - \psi_{s}^{2} \right\Vert \right] \le \left\Vert z^{1} - z^{2} \right\Vert +\\
	+ \varepsilon K \left( 6\varepsilon + M\varepsilon^{3/2} \right) \mathrm{e}^{-\varepsilon\gamma t} t \sup_{s\ge0} \left[ \mathrm{e}^{\varepsilon\gamma s} \left\Vert \zeta_{s}^{1} - \zeta_{s}^{2} \right\Vert \right] + \frac{K\varepsilon^{3/2}}{\gamma} \sup_{s\ge0} \left[ \mathrm{e}^{\varepsilon\gamma s} \left\Vert \psi_{s}^{1} - \psi_{s}^{2} \right\Vert \right] \le\\
	\le \left[ 1 + \frac{3M\varepsilon_{0}}{\mathrm{e}} + \frac{M^{2} \varepsilon_{0}^{3/2}}{\mathrm{2e}} + \frac{M^{2}\varepsilon_{0}^{1/2}}{4} \right] \left\Vert z^{1} - z^{2} \right\Vert.
\end{gathered}
\end{equation*}
If we choose $\varepsilon_{0}>0$ that satisfies inequalities~\eqref{eq:ineq1} and
\begin{equation}\label{eq:ineq2}
	\frac{3M\varepsilon_{0}}{\mathrm{e}} + \frac{M^{2}\varepsilon_{0}^{3/2}}{\mathrm{2e}} + \frac{M^{2}\varepsilon_{0}^{1/2}}{4} \le 1,
\end{equation}
and also include that the improper integral in the expression for $\mathcal{F}$ converges uniformly in $y ,z, \varphi$, then for all $\varepsilon \in (0, \varepsilon_{0})$ and $(\zeta, \psi) \in \mathfrak{M}_{\varepsilon}$ we will have
\begin{equation}\label{eq:map_F_G}
	\mathfrak{M}_{\varepsilon} \ni (\zeta,\psi) \mapsto \left( \mathcal{G} [\zeta, \psi], \mathcal{F} [\zeta, \psi] \right) \in \mathfrak{M}_{\varepsilon}.
\end{equation}

Next, having defined for every $(\zeta^{1}, \psi^{1}), (\zeta^{2}, \psi^{2}) \in \mathfrak{M}_{\varepsilon}$ the functions  $\zeta_{t}^{i} = \zeta^{i}(t,y,z,\varphi)$, $\psi_{t}^{i} = \psi^{i}(t,y,z,\varphi)$ and having set in inequalities~\eqref{eq:ineq_om+H} and~\eqref{eq:ineq_F_G} $(y^{i}, \varphi^{i}) = (\zeta_{s}^{i}, \psi_{s}^{i})$, $i = 1, 2$, $y^{3} = \eta_{s}$, $\varphi^{3} = \phi_{s}$, we draw in entirely the same manner as before that
\begin{equation*}
\begin{gathered}
	\mathrm{e}^{\varepsilon\gamma t} \left[ \left\Vert \mathcal{F} [\zeta^{1}, \psi{}^{1}] (t) - \mathcal{F} [\zeta^{2}, \psi^{2}] (t) \right\Vert \right] \le\\
	\le \frac{M}{4\varepsilon} \sup_{\mathcal{U}_{\varepsilon}} \left[ \mathrm{e}^{\varepsilon\gamma s} \left\Vert \zeta_{s}^{1} - \zeta_{s}^{2} \right\Vert \right] + \frac{M\varepsilon_{0}^{3/2}}{4} \sup_{\mathcal{U}_{\varepsilon}} \left[ \mathrm{e}^{\varepsilon\gamma s} \left\Vert  \psi_{s}^{1} - \psi_{s}^{2} \right\Vert \right],\\
	\mathrm{e}^{\varepsilon\gamma t} \left[ \left\Vert \mathcal{G} [\zeta^{1}, \psi{}^{1}] (t) - \mathcal{G}_{t}[\zeta^{2}, \psi^{2}] (t) \right\Vert \right] \le\\
	\le \varepsilon K \left( 6\varepsilon + M\varepsilon^{3/2} \right) \mathrm{e}^{-\varepsilon\gamma t} t \sup_{\mathcal{U}_{\varepsilon}} \left[ \mathrm{e}^{\varepsilon\gamma s} \left\Vert  \zeta_{s}^{1} - \zeta_{s}^{2} \right\Vert \right] + \frac{K\varepsilon^{3/2}}{\gamma} \sup_{\mathcal{U}_{\varepsilon}} \left[ \mathrm{e}^{\varepsilon\gamma s} \left\Vert \psi_{s}^{1} - \psi_{s}^{2} \right\Vert \right] \le\\
	\le \frac{M\varepsilon_{0}}{4\mathrm{e}} \left( 6 + M\varepsilon_{0}^{1/2} \right) \sup_{\mathcal{U}_{\varepsilon}} \left[ \mathrm{e}^{\varepsilon\gamma s} \left\Vert \zeta_{s}^{1} - \zeta_{s}^{2} \right\Vert \right] + \frac{M\varepsilon\varepsilon_{0}^{1/2}}{4} \sup_{\mathcal{U}_{\varepsilon}} \left[\mathrm{e}^{\varepsilon\gamma s} \left\Vert \psi_{s}^{1} - \psi_{s}^{2} \right\Vert \right].
\end{gathered}
\end{equation*}
It yields
\begin{equation*}
\begin{gathered}
	d \left[ \left( \mathcal{G} [\zeta^{1}, \psi^{1}], \mathcal{F} [\zeta^{1}, \psi^{1}] \right), \left(\mathcal{G} [\zeta^{2}, \psi^{2}], \mathcal{F} [\zeta^{2}, \psi^{2}] \right) \right] \le\\
	\le \frac{M\varepsilon_{0}^{3/2} + M^{2}\varepsilon_{0}^{1/2}}{4} \varepsilon \sup_{\mathcal{U}_{\varepsilon}} \left[ \mathrm{e}^{\varepsilon\gamma s} \left\Vert \psi_{s}^{1} - \psi_{s}^{2} \right\Vert \right] + \left[ \frac{1}{4} + \frac{\varepsilon_{0} M\left( 6 + M\varepsilon_{0}^{1/2}\right)}{4\mathrm{e}} \right] M \sup_{\mathcal{U}_{\varepsilon}} \left[ \mathrm{e}^{\varepsilon\gamma s} \left\Vert \zeta_{s}^{1} - \zeta_{s}^{2} \right\Vert \right].
\end{gathered}
\end{equation*}
One can easily ensure, that if $\varepsilon_{0}$ is small enough, inequalities~\eqref{eq:ineq1},~\eqref{eq:ineq2} grant the validity of the contraction conditions in $\mathfrak{M}_{\varepsilon}$ for mapping~\eqref{eq:map_F_G} for all $\varepsilon \in (0, \varepsilon_{0})$. The fixed point $\left(\zeta, \psi\right)$ of this mapping is precisely the solution of system~\eqref{eq:sys_phi_y} which belongs to the space $\mathfrak{M}_{\varepsilon}$.

By~\eqref{eq:def_phi0} we can explicitly determine
\begin{equation*}
	\theta(y, z, \varphi) := \psi(0, y, z, \varphi).
\end{equation*}
Clearly, the mapping $\theta(\cdot, \cdot, \cdot)$ is continuous in all variables and is Lipschitz-continuous in $z$ with the  constant $M / \varepsilon$. Moreover, since $\zeta(0, y, z, \varphi) = z$, we have
\begin{equation*}
	\zeta(t, y, z, \varphi) = \eta_{t} \left( z, \theta(y, z, \varphi) \right), \quad
	\psi(t, y, z, \varphi) = \phi_{t} \left( z, \theta(y, z, \varphi) \right).
\end{equation*}

Finally, it follows from the very construction of the mapping $\vartheta_{y}^{z}(\cdot) := \theta(y, z, \cdot)$ that for all  $(y, z, \varphi) \in \mathcal{B}_{\varepsilon} \times \mathbb{T}^{n}$ the equality $\vartheta_{z}^{y} \circ \vartheta_{y}^{z}(\varphi) = \varphi$ holds. Indeed, by~\eqref{eq:ineq_varphi-phi} to the point $\left (z, y, \varphi^{\prime} \right) \in \mathcal{B}_{\varepsilon} \times \mathbb{T}^{n}$, where $\varphi^{\prime} = \vartheta_{y}^{z}(\varphi)$, there is the corresponding point $\vartheta(z, y, \varphi^{\prime}) = \varphi$. Thus, there is the continuous inverse mapping $\left[ \vartheta_{y}^{z} \right]^{-1}(\cdot) = \vartheta_{z}^{y}(\cdot)$ defined on the set $\vartheta_{y}^{z} \left( \mathbb{T}^{n} \right)$. But then $\vartheta_{y}^{z} \left( \mathbb{T}^{n} \right)$ being an open-closed subset of a torus, which itself is an open-closed set, coincides with $\mathbb{T}^{n}$. Hence, $\vartheta_{y}^{z}(\cdot)$ is a homeomorphism of the torus onto itself.
\end{proof}

Since $\det F_{y}^{\prime} (y^{\ast}, 0) \ne 0$, the implicit function theorem says that for sufficiently small $\varepsilon_{0} > 0$, there exists a unique smooth mapping $y_{\ast}(\cdot) \colon [0, \varepsilon_{0}] \to \mathbb{R}^{n+m}$ such that  $y_{\ast}(0) = y^{\ast}$ and $F(y_{\ast}(\varepsilon), \varepsilon) = 0$ for all $\varepsilon \in [0, \varepsilon_{0}]$. On the ground of invariant tori perturbation theory~\cite{Hale61,Sam91,Sam97,Fen71,Bib90} let us prove the next proposition.

\begin{proposition}\label{prop:lipinvtor}
There exists such $\varepsilon_{0} > 0$ that for all $\varepsilon \in (0, \varepsilon_{0})$ system~\eqref{eq:sys_phi_y} has an invariant torus $\mathcal{T}_{\varepsilon}$ given by equation $y = y_{\ast}(\varepsilon) + \varepsilon \xi_{\varepsilon}(\varphi)$, where the mapping $\xi_{\varepsilon}(\cdot) \colon \mathbb{T}^{n} \to B_{\varrho(\varepsilon)}^{n+m}(0)$ satisfies the Lipschitz condition with the  Lipschitz constant $L(\varepsilon)$ such that $L(\varepsilon) \to 0$, and $\varrho(\varepsilon) \to 0$ when $\varepsilon \to 0$. This torus is a local attractor and it attracts all forward trajectories which start from an $\varepsilon( 1 - \varrho(\varepsilon))$-neighborhood of the point $y_{\ast}(\varepsilon)$.
\end{proposition}

\begin{proof}
Let us switch to a new variable $\xi$ in system~\eqref{eq:sys_phi_y} with the substitution $y = y_{\ast}(\varepsilon) + \varepsilon\xi$. It will give us
\begin{equation*}
\begin{aligned}
	\dot{\xi} &= \varepsilon \left[ F^{\prime}(y_{\ast}(\varepsilon), \varepsilon) \xi + \sqrt{\varepsilon}G(y_{\ast}(\varepsilon) + \varepsilon \xi, \varphi, \varepsilon) \right],\\
	\dot{\varphi} &= \bar{\omega}(y_{\ast}(\varepsilon) + \varepsilon \xi, \varepsilon) + \varepsilon^{5/2}H(y_{\ast}(\varepsilon) + \varepsilon\xi, \varphi, \varepsilon).
\end{aligned}
\end{equation*}
This system can be written as
\begin{equation}\label{eq:sys_phi_xi}
\begin{aligned}
	\dot{\xi} &= \varepsilon \left[ F_{y}^{\prime}(y^{\ast}, 0) \xi + \Xi(\xi, \varphi, \varepsilon) \right],\\
	\dot{\varphi} &= \bar{\omega}(y_{\ast}(\varepsilon), \varepsilon) + \varepsilon \Theta(\xi, \varphi, \varepsilon),
\end{aligned}
\end{equation}
where
\begin{equation*}
\begin{gathered}
	\Xi(\varphi, \xi, \varepsilon) = \left[F_{y}^{\prime}(y_{\ast}(\varepsilon), \varepsilon) - F_{y}^{\prime}(y^{\ast}, 0) \right] \xi + \sqrt{\varepsilon}G(y_{\ast}(\varepsilon) + \varepsilon \xi, \varphi, \varepsilon),\\
	\Theta(\varphi, \xi, \varepsilon) = \intop_{0}^{1} \bar{\omega}_{y}^{\prime} \left( y_{\ast}(\varepsilon) + \varepsilon s \xi, \varepsilon \right) \xi \mathrm{d} s + \varepsilon^{3/2}H(y_{\ast}(\varepsilon) + \varepsilon \xi, \varphi, \varepsilon).
\end{gathered}
\end{equation*}
One can easily make sure that Lemma~2.1 from~\cite{Hale61} is applicable to system~\eqref{eq:sys_phi_xi}. According to this lemma, there exists $\varepsilon_{0} > 0$ such that for all $\varepsilon \in (0, \varepsilon_{0})$ the system has an invariant torus defined by equation $\xi = \xi_{\varepsilon}(\varphi)$, where the mapping $\xi_{\varepsilon}(\cdot)$ possesses all of the properties mentioned earlier.

Hence, system~\eqref{eq:sys_phi_y} in an $\varepsilon\varrho(\varepsilon)$-neighborhood of point $y_{\ast}(\varepsilon)$ has an invariant torus $\mathcal{T}_{\varepsilon}$ given by equation $y = y_{\ast} (\varepsilon) + \varepsilon\xi_{\varepsilon}(\varphi)$. Let us assign $z_{\varepsilon}(\varphi) := y_{\ast}(\varepsilon) + \varepsilon\xi_{\varepsilon}(\varphi)$. For sufficiently small $\varepsilon_{0}$ and $\varepsilon \in (0,\varepsilon_{0})$ the point $y_{\ast}(\varepsilon)$ lies in $B_{\sigma}^{n+m}(y^{\ast})$ together with its $\varepsilon$-neighborhood and for any such $y_{0}$ that $\left\Vert y_{0} - y_{\ast}(\varepsilon) \right\Vert < \varepsilon( 1 - \varrho(\varepsilon))$ and arbitrary $\varphi \in \mathbb{T}^{n}$ the equalities
\begin{equation*}
	\left\Vert y_{0} - z_{\varepsilon}(\varphi) \right\Vert \le \left\Vert y_{0} - y_{\ast}(\varepsilon) \right\Vert + \left\Vert y_{\ast}(\varepsilon) - z_{\varepsilon}(\varphi) \right\Vert < \varepsilon(1 - \varrho(\varepsilon)) + \varepsilon\varrho(\varepsilon) = \varepsilon.
\end{equation*}
are valid. Thus, if $y_{0}$ is an arbitrary point of an $\varepsilon(1 - \varrho(\varepsilon))$-neighborhood of the point $y_{\ast}(\varepsilon)$, then $(y_{0}, z_{\varepsilon}(\varphi)) \in \mathcal{B}_{\varepsilon}$ for all $\varphi \in \mathbb{T}^{n}$. Proposition~\ref{prop:def_theta} implies that for all $t\ge0$ we have
\begin{equation*}
\begin{aligned}
	\left\Vert \eta_{t} \left(y_{0}, \varphi_{0} \right) - \eta_{t} \left( z_{\varepsilon}(\varphi), \theta(y_{0}, z_{\varepsilon}(\varphi), \varphi_{0}) \right) \right\Vert & \le 2 \mathrm{e}^{-\varepsilon\gamma t} \left\Vert y_{0} - z_{\varepsilon}(\varphi) \right\Vert,\\
	\left\Vert \phi_{t} \left(y_{0}, \varphi_{0} \right) - \phi_{t} \left( z_{\varepsilon}(\varphi), \theta(y_{0}, z_{\varepsilon}(\varphi), \varphi_{0}) \right) \right\Vert & \le \frac{M}{\varepsilon}\mathrm{e}^{-\varepsilon\gamma t} \left\Vert y_{0} - z_{\varepsilon}(\varphi) \right\Vert.
\end{aligned}
\end{equation*}
It means that the forward trajectory of the point $(y_{0},\varphi_{0})$ is attracted to the forward trajectory of the point $(z_{\varepsilon}(\varphi), \theta(y_{0},z_{\varepsilon}(\varphi), \varphi_{0}))$. To ensure that the latter trajectory lies on the invariant torus, the fixed point condition has to be fulfilled:
\begin{equation*}
	\theta(y_{0},z_{\varepsilon}(\varphi), \varphi_{0}) = \varphi.
\end{equation*}
Let us show that such a fixed point on the torus $\mathbb{T}^{n}$ exists. The number $\varepsilon_{0}$ may be taken small enough for $ML(\varepsilon) < 1$ to hold for all $\varepsilon\in(0,\varepsilon_{0})$. Then, in accordance with Proposition~\ref{prop:def_theta}, for arbitrary $\varphi_{1}, \varphi_{2} \in \mathbb{T}^{n}$ we have
\begin{equation*}
	\left\Vert \theta(y_{0},z_{\varepsilon}(\varphi_{1}),\varphi_{0})-\theta(y_{0},z_{\varepsilon}(\varphi_{2}),\varphi_{0})\right\Vert \le ML(\varepsilon)\left\Vert \varphi_{1}-\varphi_{2}\right\Vert .
\end{equation*}
Consequently, the contraction mappings principle implies the existence of the unique point $\varphi_{\ast} = \varphi_{\ast}(y_{0},\varphi_{0}) \in \mathbb{T}^{n}$ such that  $\theta(y_{0}, z_{\varepsilon}(\varphi_{\ast}) ,\varphi_{0}) = \varphi_{\ast}$, which means that the forward trajectory of the point  $(y_{0}, \varphi_{0})$ is attracted to the forward trajectory of the point $(z_{\varepsilon}(\varphi_{\ast}), \varphi_{\ast})$ of the invariant torus  $\mathcal{T}_{\varepsilon}$. At the same time
\begin{equation}\label{eq:attractT_eps}
	\left\Vert \eta_{t}(y_{0}, \varphi_{0}) - z_{\varepsilon} \left( \phi_{t}(\varphi_{\ast}) \right) \right\Vert \le 2 \mathrm{e}^{-\varepsilon\gamma t} \left\Vert y_{0} - z_{\varepsilon}(\varphi_{\ast}) \right\Vert, \quad t \ge 0.
\end{equation}
Let us note that since $\varphi_{\ast}$ can be found using the method of subsequent approximations, $\varphi_{\ast}(y_{0},\varphi_{0})$ continuously depends on $(y_{0},\varphi_{0})$. It follows from~\eqref{eq:ineq_varphi-phi} for $t=0$ that $\theta(z_{\varepsilon}(\varphi_{0}), z_{\varepsilon}(\varphi_{0}), \varphi_{0}) = \varphi_{0}$,
and therefore, $\varphi_{\ast}(z_{\varepsilon}(\varphi_{0}), \varphi_{0}) = \varphi_{0}$. Thus, $\varphi_{\ast}(y_{0}, \varphi_{0}) \to\varphi_{0}$ when $y_{0} \to z_{\varepsilon}(\varphi_{0})$. But then at the same time $z_{\varepsilon}(\varphi_{\ast}(y_{0}, \varphi_{0})) \to y_{0}$. By~\eqref{eq:attractT_eps} this yields that for any $\Delta>0$ there exists such $\delta>0$, that the forward trajectory of a point which lies in a $\delta$-neighborhood of the torus $\mathcal{T}_{\varepsilon}$ belongs to a $\Delta$-neighborhood of this torus and is attracted to the latter. This does mean that $\mathcal{T}_{\varepsilon}$ is a local attractor.
\end{proof}

\begin{remark}
In case of a quasi-periodic flow on an invariant torus, an estimate similar to~\eqref{eq:attractT_eps} was obtained in~\cite{Sam97}.
\end{remark}

\begin{proposition}
The attraction basin of the invariant torus $\mathcal{T}_{\varepsilon}$ contains the forward invariant set $\left[ V_{0}^{-1} \left( [0, \left| \ln\varepsilon^{k} \right| ] \right) \cap \mathcal{S}_{\varrho} \right] \times  B_{R^{\ast}}^{m} \times \mathbb{T}^{n}$.
\end{proposition}

\begin{proof}
Turning back to Propositions~\ref{prop:glob_sol_nf} and~\ref{prop:moment_t_eps}, it is enough to prove that $B_{\sigma}^{n+m}(y^{\ast}) \times \mathbb{T}^{n}$ lies in the attraction basin of the torus $\mathcal{T}_{\varepsilon}$. Suppose that $(y_{0},\varphi_{0})$ is an arbitrary point of the domain $B_{\sigma}^{n+m}(y^{\ast}) \times \mathbb{T}^{n}$. Let us choose a finite series of points $\left\{ y_{i} \right\}_{i=1}^{I}$ such that $\left\Vert y_{i-1} - y_{i} \right\Vert <\varepsilon$ with the last point $y_{I}$ lying in an $\varepsilon(1 - \varrho(\varepsilon)) $-neighborhood of the point $y_{\ast}(\varepsilon)$. Then, applying Proposition~\ref{prop:def_theta} step by step, we can prove that there exists such $\varphi_{I}\in\mathbb{T}^{n}$ that the forward trajectory of the point $(\varphi_{0},y_{0})$ is attracted to the forward trajectory of the point $(\varphi_{I},y_{I})$, which in turn, according to Proposition~\ref{prop:lipinvtor}, is attracted to the torus $\mathcal{T}_{\varepsilon}$. Thus, under action of the semi-flow of system~\eqref{eq:sys_phi_y} $(\varphi_{0},y_{0})$ is attracted to the torus $\mathcal{T}_{\varepsilon}$.
\end{proof}

\section{Conclusion}

In this paper, we have analyzed a kind of transient processes that are observable in a fast-slow system in a neighborhood of an i.~m.~s.~m. and that can be interpreted as a dynamical bifurcation of multi-frequency oscillations. The change in the phase variables $x(t)$ behavior --- the switch from damping oscillations to the multi-frequency ones, which are asymptotically close to motions on the invariant torus $\mathcal{T}_{\varepsilon}$ --- is caused by the slow evolution of the parameters $u(t)$, which results in transition of the latter from the stability zone $\mathcal{D}_{s}$ to the zone of complete instability $\mathcal{D}_{u}$.

 We should note that there is certain connection between the obtained results and the theory of bifurcations without parameters (see~\cite{Lie11} and references there). In~\cite{Lie11} system~\eqref{eq:slman1} was investigated in the case where the i.~m.~s.~m. $x=0$ consists completely of equilibria and the spectrum of the operator $f_{x}^{\prime}(0,0,\varepsilon)$ lies on the imaginary axis. In this situation the transformations of the phase portrait are caused by different kinds of hyperbolicity of the system of the first approximation with respect to points of the i.~m.~s.~m. $(0,u)$ if $u\ne0$.

\end{document}